\documentclass{article}   	
\usepackage{geometry}     
\usepackage{xfrac}
\usepackage{graphicx}
\usepackage[usenames,dvipsnames]{xcolor}
\usepackage{amssymb,amsmath,amsthm,amsfonts}

\allowdisplaybreaks

\usepackage[english]{babel}
\usepackage[utf8]{inputenc}
\usepackage[colorlinks]{hyperref}
\hypersetup{linkcolor=blue,citecolor=blue,filecolor=black,urlcolor=blue}

\usepackage{mathtools}
\usepackage{esint} 
\usepackage{dsfont}

\usepackage[sc]{mathpazo}

\usepackage{tikz}
\newtheorem{proposition}{Proposition}[section]
\newtheorem{theorem}[proposition]{Theorem}
\newtheorem{corollary}[proposition]{Corollary}
\newtheorem{lemma}[proposition]{Lemma}

\theoremstyle{definition}
\newtheorem{definition}[proposition]{Definition}
\theoremstyle{remark}
\newtheorem{remark}[proposition]{Remark}
\numberwithin{equation}{section}

\newcommand{\N}{{\mathbb{N}}}

\newcommand{\R}{{\mathbb{R}}}

\newcommand{\loc}{{\mathrm{loc}}}

\DeclareMathOperator{\dist}{dist}

\DeclareMathOperator{\diverg}{div}

\title{A one-sided two phase Bernoulli free boundary problem}
\author{Lorenzo Ferreri, Bozhidar Velichkov}

\begin{document}

\maketitle

\begin{abstract}
We study a two-phase free boundary problem in which the two-phases satisfy an impenetrability condition. Precisely, we have two ordered positive functions, which are harmonic in their supports, satisfy a Bernoulli condition on the one-phase part of the free boundary and a two-phase condition on the collapsed part of the free boundary. For this two-membrane type problem, we prove an epsilon-regularity theorem with sharp modulus of continuity. Precisely, we show that at flat points each of the two boundaries is $C^{1,\sfrac12}$ regular surface. Moreover, we show that the remaining singular set has Hausdorff dimension at most $N-5$ as in the case of the classical one-phase problem, $N$ being the dimension of the space.
\end{abstract}

\noindent
{\footnotesize \textbf{AMS-Subject Classification} 35R35}. 

\tableofcontents

\section{Introduction}\label{section:intro}
\subsection{Bernoulli free boundary problems - an overview}
The Bernoulli-type free boundary problems arise from models in fluid dynamics. The classical one-phase problem 
\begin{equation*}
\begin{cases}
\Delta u = 0 & \text{in } \{ u>0 \} \cap B_1 , \\
\vert \nabla u \vert^2 = 1 & \text{on } \partial\{ u>0 \} \cap B_1 \\
u \ge  0 & \text{in } B_1 ,\\
u = u_0 & \text{on } \partial B_1 ,
\end{cases}
\end{equation*}
was first studied in the seminal paper of Alt and Caffarelli \cite{AltCaffarelli:OnePhaseFreeBd} in 1981 and had a fundamental role in the development of the free boundary regularity theory since then. It is today known that if $u$ is a variational solution of this problem, that is, a minimizer of the functional 
$$\int_{B_1}|\nabla u|^2\,dx+|\{u>0\}\cap B_1|,$$
then the free boundary $\partial\{u>0\}$ can be decomposed into a regular and a singular parts, the regular part being $C^{1,\alpha}$ smooth manifold (thanks to \cite{AltCaffarelli:OnePhaseFreeBd} and \cite{DeSilva:FreeBdRegularityOnePhase}), while the singular part is a closed set of Hausdorff dimension at most $N-{N^\ast}$, where $N^\ast$ is the first dimension in which singular free boundaries appear (see \cite{Weiss99:PartialRegularityFreeBd}) and it is known that $5\le N^\ast\le 7$ (\cite{DeSilvaJerison09:SingularConesIn7D}, \cite{CaffarelliJerisonKenig04:NoSingularCones3D} and \cite{JerisonSavin15:NoSingularCones4D}). 

The two-phase counterpart of this problem was proposed by Alt, Caffarelli and Friedman \cite{AltCaffarelliFriedman1984:TwoPhaseBernoulli} in 1984 and consists in minimizing the functional
\begin{equation}\label{e:two-phase-functional-introduction}
\int_{B_1}|\nabla u|^2\,dx+\int_{B_1}|\nabla v|^2\,dx+\Lambda_u|\{u>0\}\cap B_1|+\Lambda_v|\{v>0\}\cap B_1|,
\end{equation}
with $\Lambda_u>0$, $\Lambda_v>0$, and under the condition 
$$\{u>0\}\cap\{v>0\}=\emptyset\quad\text{in}\quad B_1.$$
This leads to the system 
\begin{equation*}
\begin{cases}
\Delta u = 0 & \text{in } \{ u>0 \} \cap B_1 , \\
\Delta v = 0 & \text{in } \{ v>0 \} \cap B_1 , \\
\vert \nabla u \vert^2 = \Lambda_u & \text{on } \partial\{ u>0 \} \setminus  \partial\{ v>0 \}, \\
\vert \nabla v \vert^2 = \Lambda_v & \text{on } \partial\{ v>0 \} \setminus  \partial\{ u>0 \}, \\
\vert \nabla u \vert^2 - \vert \nabla v \vert^2 = \Lambda_u - \Lambda_v & \text{on } \partial\{ u>0 \} \cap \partial\{ v>0 \} , \\
u \ge 0, \quad v \ge 0 & \text{in } B_1 ,\\
u = u_0, \quad v = v_0 & \text{on } \partial B_1,
\end{cases}
\end{equation*}
where $u_0$ and $v_0$ are fixed boundary data. 
This problem was extensively studied and generalized in the case $|\{u=0\}\cap \{v=0\}\cap B_1|=0$, starting from the original paper \cite{AltCaffarelliFriedman1984:TwoPhaseBernoulli}; it led to important developments as the series of papers \cite{Caffarelli1987:Harnack1,Caffarelli1989:Harnack2} (see also the book \cite{CaffarelliSalsa:GeomApproachToFreeBoundary}) and the more recent \cite{DeSilvaFerrariSalsa:2PhaseFreeBdDivergenceForm, DeSilvaFerrariSalsa2015, DeSilvaFerrariSalsa2017} and the references therein. 

The regularity of the free boundary of the solutions to the two-phase problem in the general case, without the requirement $|\{u=0\}\cap \{v=0\}\cap B_1|=0$, was obtained only recently in \cite{SpolaorVelichkov} and \cite{DePhilippisSpolaorVelichkov2021:TwoPhaseBernoulli}, where it was shown that, in a neighborhood of any two-phase boundary point $x_0\in\partial\{u>0\}\cap\partial\{v>0\}$, both $\partial\{u>0\}$ and $\partial\{v>0\}$ are $C^{1,\alpha}$ manifolds. 

A more general version of the two-phase problem is the so-called vectorial Bernoulli problem introduced and studied in \cite{CaffarelliShahgholianYeressian2018:vectorial,KriventsovLin2018:nondegenerate,MazzoleniTerraciniVelichkov2017:RegularitySpectralFunctionals,KriventsovLin2019:degenerate, MazzoleniTerraciniVelichkov2020:RegularityVectorialBernoulli, DeSilvaTortone2020:ViscousVectorialBernoulli}, which consists in minimizing the functional 
$$\int_{B_1}|\nabla U|^2\,dx+|\{|U|>0\}\cap B_1|,$$
where $U=(u_1,\dots,u_k)$, $U:\R^d\to\R^k$ is a vector valued function with prescribed datum on $\partial B_1$. If we take $k=2$, $u=u_1$ and $v=u_2$, then in $B_1$ a minimizer $U=(u,v)$ formally solves the free boundary problem
\begin{equation*}
\begin{cases}
\Delta u = 0 & \text{in } \Omega \cap B_1 , \\
\Delta v = 0 & \text{in } \Omega \cap B_1 , \\
u=v=0& \text{on } B_1\setminus \Omega ,\\
\vert \nabla u \vert^2 + \vert \nabla v \vert^2 = 1 & \text{on } \partial\Omega\cap B_1 , 
\end{cases}
\end{equation*}
where $\Omega=\{|U|>0\}$. This is a particular case of a free boundary system in which several functions solve a PDE in the same domain $\Omega$ and then satisfy a boundary condition on $\partial\Omega$. We notice that for solutions of the vectorial problem singular points might occur on the boundary of $\Omega$ even in the case when all the components of the vector $U$ are non-negative (see \cite{CaffarelliShahgholianYeressian2018:vectorial,KriventsovLin2018:nondegenerate,MazzoleniTerraciniVelichkov2017:RegularitySpectralFunctionals} for the vectorial case and also \cite{buttazzo2022:regularity} for the case of another free boundary system).  \medskip

\subsection{Statement of the problem}
In the present paper we study a variational free boundary problem with two phases $u$ and $v$, in which we minimize the Alt-Caffarelli-Friedman's functional \eqref{e:two-phase-functional-introduction} under the condition $$u\ge v\ge 0\quad\text{in}\quad B_1,$$ 
which,  
 in particular, means that their positivity sets satisfy the inclusion constraint 
$$\{v>0\}\subset \{u>0\}\quad\text{in}\quad B_1.$$
The minimizers $u$ and $v$ of this variational problem are solutions, in a suitable sense (see Section \ref{subsection:ViscositySetting}), to the system
\begin{equation}\label{eqn:MainViscousPb}
\begin{cases}
\Delta u = 0 & \text{in } \{ u>0 \} \cap B_1 , \\
\Delta v = 0 & \text{in } \{ v>0 \} \cap B_1 , \\
\vert \nabla u \vert^2 = \Lambda_u & \text{on } \partial\{ u>0 \} \setminus  \partial\{ v>0 \}, \\
\vert \nabla v \vert^2 = \Lambda_v & \text{on } \partial\{ v>0 \} \setminus  \partial\{ u>0 \}, \\
\vert \nabla u \vert^2 + \vert \nabla v \vert^2 = \Lambda_u + \Lambda_v = 1 & \text{on } \partial\{ u>0 \} \cap \partial\{ v>0 \} , \\
u > 0, \quad v > 0 & \text{in } B_1 ,\\
u = u_0, \quad v = v_0 & \text{on } \partial B_1,
\end{cases}
\end{equation}
where $\Lambda_u>\Lambda_v>0$ and $u_0\ge v_0$ are given.\medskip

This problem models for instance the lowest energy equilibria of two linear elastic membranes in air (with fixed configuration at the boundary), in which the required ordering has the meaning of an impenetrability condition. \medskip

The problem \eqref{eqn:MainViscousPb} combines some of the main difficulties from the two-phase and the vectorial problems. Indeed, each of the two function $u$ and $v$ solves a PDE only on its positivity set; then, there are different boundary conditions on the one-phase free boundaries  $\partial\{ u>0 \} \setminus \partial\{ v>0 \}\cap B_1$ and $\partial\{ v>0 \} \setminus \partial\{ u>0 \}\cap B_1$, and on the two-phase free boundary $\partial\{ u>0 \} \cap \partial\{ v>0 \}\cap B_1$. In this, the problem \eqref{eqn:MainViscousPb} is similar to the two-phase problem from \cite{DePhilippisSpolaorVelichkov2021:TwoPhaseBernoulli}. In fact, as we will show in Theorem \ref{thm:C1aRegularity}, the two boundaries $\{u>0\}$ and $\{v>0\}$ are tangent at any two-phase point $x_0\in\partial\{ u>0 \} \cap \partial\{ v>0 \}\cap B_1$ and, just as in \cite{DePhilippisSpolaorVelichkov2021:TwoPhaseBernoulli}, they can separate forming a cusp. On the other hand, contrary to what happens in \cite{DePhilippisSpolaorVelichkov2021:TwoPhaseBernoulli}, $\{u>0\}$ and $\{v>0\}$ might not be smooth in a neighborhood of a two-phase point and singularities two-phase points might appear, exactly as in the case of the vectorial problem. \medskip

Finally, we notice that a Bernoulli free boundary problem with multiple phases $u_1,\dots,u_k$ satisfying the impenetrability condition
$$u_1\ge u_2\ge \dots\ge u_k\ge 0,$$
was studied recently by De Silva and Savin in \cite{DeSilvaSavin2023:MultiMembrane} (see also \cite{SavinHui2021:3membrane} and \cite{SavinHui2023:multimembrane} for the obstacle problem counterpart). We stress that, even if the impenetrability condition is the same as ours, their functional is different and leads to a free boundary system in which the functions $u_j$ are not harmonic in $\{u_j>0\}$ as their gradients jump across the junctions $\{u_j=u_{j+1}>0\}$. 

\subsection{Notations and main results} Given two positive constants $\Lambda_u \ge \Lambda_v >0$ let us denote with $\mathcal{F}$ the functional
\[
\mathcal{F}(u, v) \coloneqq \int_{B_1} \vert \nabla u \vert^2 + \vert \nabla v \vert^2 + \Lambda_u \vert \{u>0\} \cap B_1 \vert + \Lambda_v \vert \{v>0\} \cap B_1 \vert,
\]
where, without loss of generality, we assume $\Lambda_u + \Lambda_v = 1$. Let the functions $u_0, v_0 \in H^1(B_1)$ be given, such that $u_0 \ge v_0$ almost everywhere in $B_1$ and introduce the set $\mathcal{A}_0$ defined as
\[
\mathcal{A}_0 \coloneqq \left\{ (u, v) \in H_1(B_1) \times H_1(B_1) : u-u_0, v-v_0 \in H^1_0\left(\R^N\right) \right\}.
\]
We study the problem
\[
\inf_{(u, v) \in \mathcal{A}_0} \mathcal{F}(u, v), \text{ under the additional constraint } u \ge v \text{ a.e. in } B_1.
\]
In the following, we will denote with $\mathcal{A}$ the set of competitors, namely
\[
\mathcal{A} \coloneqq \left\{ (u, v) \in \mathcal{A}_0 : u \ge v \text{ a.e. in } B_1 \right\}.
\]
More precisely, we study the couples $(u, v) \in \mathcal{A}$ such that
\begin{equation}\label{eqn:MinimumPbExplicit}
\mathcal{F}(u, v) \le \mathcal{F}(h, w) , \quad \forall (h, w) \in \mathcal{A}.
\end{equation}
The fact that minimizers exist is a straightforward application of the direct method in the calculus of variations, thanks to the ordering condition ($u \ge v$) required to the competitors. 

%
%

Before we state our main results, we notice that, given a minimizer $(u,v)$ of the above variational problem, the structure of the one-phase parts of the free boundary
$$\partial\{ u>0 \} \setminus \partial\{ v>0 \}\cap B_1\qquad\text{and}\qquad \partial\{ v>0 \} \setminus \partial\{ u>0 \}\cap B_1,$$ 
is described by the classical results for the one-phase problem (see \cite{AltCaffarelli:OnePhaseFreeBd,Weiss99:PartialRegularityFreeBd,DeSilvaJerison09:SingularConesIn7D,CaffarelliJerisonKenig04:NoSingularCones3D,JerisonSavin15:NoSingularCones4D}).

In the following, we denote with $Reg(u, v)$ the regular part of the free boundary, namely the points at which the blow up is a half-plane solution (see section \ref{subsection:BlowUp}). Notice that if $x_0$ is a regular point for $\partial\{ u>0 \}$, then it is also a regular point for $\partial\{ v>0 \}$ and viceversa (see e.g. Lemma \ref{lemma:TwoPhaseBlowUpProportional}).

We have two main results. The first one concerns the almost-optimal regularity of the free boundaries $\partial\{ u>0 \}$, $\partial\{ v>0 \}$ and is contained in the following
\begin{theorem}\label{thm:C1aRegularity}
Let $(u, v)$ be a viscosity solution of problem \eqref{eqn:MainViscousPb}. Then, the two-phase part of the free boundary can be decomposed as
$$\partial\{ u>0 \}\cap\partial\{ v>0 \}\cap B_1=Reg(u, v)\cup Sing(u, v),$$
where $Reg(u, v)$ and $Sing(u, v)$ are disjoint and have the following properties: 
\begin{enumerate}
\item[(A)]
In a neighborhood of any point $x \in Reg(u, v)$,   
the sets $\partial\{ u>0 \} \cap B_1$ and $\partial\{ v>0 \} \cap B_1$ are  $C^{1, \alpha}$ manifolds, for all $0 < \alpha < 1/2$. 
\item[(B)] There exists a critical dimension\footnote{The critical dimension $N^\ast$ is the first dimension in which there are solutions to the classical one-phase Bernoulli problem with singular free boundary.} $N^{\ast} \in \{ 5, 6, 7 \}$ such that:
\begin{itemize}
    \item[(i)] if $N < N^{\ast}$ then $Sing(u, v) = \emptyset$,

    \item[(ii)] if $N = N^{\ast}$ then $Sing(u, v)$ consists of isolated points,

    \item[(ii)] if $N > N^{\ast}$ then $\mathcal{H}^{N-N^{\ast} + s}\left( Sing(u, v) \right) = 0$ for all $s>0$.
\end{itemize}
\end{enumerate}
\end{theorem}
The whole section \ref{section:Regularity} is dedicated to the proof of Theorem \ref{thm:C1aRegularity}. The second result is a refinement of Theorem \ref{thm:C1aRegularity}, which deals with the sharp regularity for the free boundaries , and is studied in section \ref{section:SharpRegularity}.
\begin{theorem}\label{thm:SharpRegularity}
Let (u, v) be a viscosity solution of problem \eqref{eqn:MainViscousPb}. Then,
the sets $\partial\{ u>0 \} \cap B_1$ and $\partial\{ v>0 \} \cap B_1$ are of class $C^{1, 1/2}$ in a neighborhood of any $x \in Reg(u, v)$. 
\end{theorem}


The outline of the proofs of these two main theorems is given in the beginning of the sections \ref{section:Regularity} and \ref{section:SharpRegularity}.\medskip

Finally, we remark that the ordering condition $u\ge v$ is only used to guarantee existence and non degeneracy for the variational problem \eqref{eqn:MinimumPbExplicit}, while all the analysis of the free boundary regularity is carried out only under the weaker assumptions 
\begin{equation}\label{eqn:WeakOrderingCond}
\{ v>0 \} \subseteq \{ u>0 \} \text{ a.e. in } B_1 \quad \text{and} \quad \Lambda_u, \Lambda_v >0 \text{ (not necessarily ordered)}.
\end{equation}
%

\section{Preliminaries}\label{section:preliminaries}
In this section we recall the main properties of solutions $u, v$ to the variational problem \eqref{eqn:MinimumPbExplicit}. More precisely, in section \ref{subsection:LipschitzAndNonDegenerate} we deduce the local Lipschitz regularity and non-degeneracy properties of $u$ and $v$ while in section \ref{subsection:BlowUp} we study the blow-up limits of the solutions at their free boundary points, in the spirit of \cite{AltCaffarelli:OnePhaseFreeBd}.

Then, in section \ref{subsection:ViscositySetting} we introduce the definition of viscosity solution for problem \eqref{eqn:MainViscousPb} and we show that variational solutions of problem \eqref{eqn:MinimumPbExplicit} are actually viscous solutions of problem \eqref{eqn:MainViscousPb}.

\subsection{Lipschitz regularity and non-degeneracy}\label{subsection:LipschitzAndNonDegenerate}
To begin with, we recall the following property.
\begin{lemma}\label{lemma:SolutionsAreSubharmonic}
    Let $u$ and $v$ solutions to problem \eqref{eqn:MinimumPbExplicit}. Then, in $B_1$ they are subharmonic in the sense of distributions.
\end{lemma}
For the proof see e.g. \cite[Lemma 2.6]{Velichkov:RegularityOnePhaseFreeBd}. As a consequence of Lemma \ref{lemma:SolutionsAreSubharmonic} we have that $\Delta u$ and $\Delta v$ define positive Radon measures on $B_1$. We proceed with the following laplacian estimates.
\begin{lemma}\label{lemma:LaplacianEstimates}
Let $u$ and $v$ solutions to problem \eqref{eqn:MinimumPbExplicit}. Then, there exists a universal constant $C > 0$ such that
\begin{equation}\label{eqn:LaplacianEstimateThesis}
    \Delta u \left( B_r(x) \right) \le C r^{N-1} \quad \text{and} \quad \Delta v \left( B_r(x) \right) \le C r^{N-1}
\end{equation}
for every ball $B_r(x)$ with $B_{2r}(x) \subseteq B_1$.
\end{lemma}
\begin{proof}
    Proceeding as in \cite[Lemma 3.9]{Velichkov:RegularityOnePhaseFreeBd} one can show that
    \[
    \left( \Delta u + \Delta v \right)\left( B_r(x) \right) \le C r^{N-1}.
    \]
    Since $\Delta u$ and $\Delta v$ are both positive Radon measures, \eqref{eqn:LaplacianEstimateThesis} follows.
\end{proof}
As a corollary of Lemma \ref{lemma:LaplacianEstimates}, we have the following local Lipschitz regularity result.
\begin{corollary}\label{crl:localLipschitzReg}
    Let $u$ and $v$ solutions to problem \eqref{eqn:MinimumPbExplicit}. Then, they are locally Lipschitz continuous in $B_1$.
\end{corollary}
For the proof of Corollary \ref{crl:localLipschitzReg} we refer e.g. to \cite[Section 3]{Velichkov:RegularityOnePhaseFreeBd}. Now we turn to the non-degeneracy property.
\begin{lemma}\label{lemma:SolutionsNonDegeneracy}
    Let $u$ and $v$ solutions to problem \eqref{eqn:MinimumPbExplicit}. Then, there exists a constant $c>0$ depending only on $\Lambda_u$, $\Lambda_v$ and the dimension $N$ such that
\begin{align}
    & \| u \|_{L^{\infty}(B_r(x))} \ge cr \quad \text{for all } x \in  \overline{\Omega_u} \cap B_1 \text{ and } B_r(x) \subseteq B_1, \label{eqn:NonDegeneracyUThesis}\\
    & \| v \|_{L^{\infty}(B_r(y))} \ge cr \quad \text{for all } y \in  \overline{\Omega_v} \cap B_1 \text{ and } B_r(x) \subseteq B_1. \label{eqn:NonDegeneracyVThesis}
\end{align}
\end{lemma}
\begin{proof}
    Since $\{ v>0 \} \subseteq \{ u>0 \}$, one can proceed exactly as in \cite[Lemma 4.4]{Velichkov:RegularityOnePhaseFreeBd} to show \eqref{eqn:NonDegeneracyVThesis}. We are left to show \eqref{eqn:NonDegeneracyUThesis}. To this aim let $x \in \partial \{ u>0 \}$, $d \coloneqq \dist(x, \overline{\Omega_v})$ and $y \in \partial \{ v>0 \}$ attaining $d$. Then, for all $r \in (0, d)$ one can still proceed as in \cite[Lemma 4.4]{Velichkov:RegularityOnePhaseFreeBd} and deduce \eqref{eqn:NonDegeneracyVThesis} for all $r \in (0, 2d)$ with $c/2$. For $r\ge 2d$ on can use the condition $u \ge v$ in $B_1$, together with \eqref{eqn:NonDegeneracyVThesis} at $y$.
\end{proof}

\subsection{Blow-up}\label{subsection:BlowUp}
\begin{definition}\label{def:BlowUpLimit}
    Let the functions $u$ and $v$ be solutions of problem \eqref{eqn:MinimumPbExplicit}. Suppose that $x_0 \in \partial \{ u>0 \}$ and let $r_n \to 0$ be a sequence of positive real numbers. Introduce the sequence of functions
    \[
    u_{x_0, r_n}(x) \coloneqq \frac{1}{r_n}(x_0 + r_n x) \quad \text{with } x \text{ such that } x_0 + r_n x \in B_1.
    \]
    We say that a function $u_{\infty}  \in L_{\loc}^{\infty}(\R^N) \cap H_{\loc}^1(\R^N)$ is a blow-up limit at $x_0$ for $u$, if there exists a sequence $0<r_n \to 0$ such that
    \[
    \left\| u_{x_0, r_n}(x) - u_{\infty}(x) \right\|_{L^{\infty}(B_R)} \to 0 \quad \text{for any fixed } R>0.
    \]
    An analogous definition can be given for a blow up limit of $v$ at $y_0 \in \partial \{ u>0 \}$.
\end{definition}
\begin{lemma}\label{lemma:BlowUpProperties}
    Let the functions $u$ and $v$ be solutions of problem \eqref{eqn:MinimumPbExplicit} and $x \in \partial \{ u>0 \} \cup \partial \{ v>0 \}$. Then:
    \begin{itemize}
        \item[(i)] If $x \in \partial \{ u>0 \} \setminus \partial \{ v>0 \}$ there exists a blow-up limit $u_{\infty}$ for $u$ at $x$, which is a solution of the following one-phase Bernoulli problem in $\R^N$
        \[
        \int_{B_R} \vert \nabla u_{\infty} \vert^2 + \Lambda_u \vert \{u_{\infty}>0\} \cap B_R \vert \le \int_{B_R} \vert \nabla (u_{\infty} + \varphi) \vert^2 + \Lambda_u \vert \{u_{\infty} + \varphi >0\} \cap B_R \vert,
        \]
        for all $R >0$ and $\varphi \in H_0^1(B_R)$.

        \item[(ii)] If $x \in \partial \{ v>0 \} \setminus \partial \{ u>0 \}$ there exists a blow-up limit $v_{\infty}$ for $v$ at $x$, which is a solution of the following one-phase Bernoulli problem in $\R^N$
        \[
        \int_{B_R} \vert \nabla v_{\infty} \vert^2 + \Lambda_v \vert \{v_{\infty}>0\} \cap B_R \vert \le \int_{B_R} \vert \nabla (v_{\infty} + \varphi) \vert^2 + \Lambda_v \vert \{v_{\infty} + \varphi >0\} \cap B_R \vert,
        \]
        for all $R >0$ and $\varphi \in H_0^1(B_R)$.

        \item[(iii)] If $x \in \partial \{ u>0 \} \cap \partial \{ v>0 \}$ there exist blow-up limits $u_{\infty}$ and $v_{\infty}$ for $u$ and $v$ at $x$, which are solutions of
        \begin{align}\label{eqn:BlowUpCoincidencePb}
        \begin{split}
        & \int_{B_R} \vert \nabla u_{\infty} \vert^2 + \vert \nabla v_{\infty} \vert^2 + \Lambda_u \vert \{u_{\infty}>0\} \cap B_R \vert + \Lambda_v \vert \{v_{\infty}>0\} \cap B_R \vert \le \\
         &  \int_{B_R} \vert \nabla (u_{\infty} + \varphi)\vert^2 + \vert \nabla (v_{\infty} + \psi )\vert^2 + \Lambda_u \vert \{u_{\infty}+ \varphi>0\} \cap B_R \vert + \Lambda_v \vert \{v_{\infty}+ \psi>0\} \cap B_R \vert
        \end{split}
        \end{align}
        for all $R >0$ and $\varphi$, $\psi \in H_0^1(B_R)$ satisfying $\{v_{\infty}+ \psi>0\} \cap B_R \subseteq \{u_{\infty}+ \varphi>0\} \cap B_R$.
    \end{itemize}
\end{lemma}
For the proof of Lemma \ref{lemma:BlowUpProperties} we refer e.g. to \cite{Velichkov:RegularityOnePhaseFreeBd}. Now we give a property of the gradient of plane blow-ups. 
\begin{lemma}\label{lemma:flatBlowUpsRightGradient}
   Let the functions $u$ and $v$ be solutions of problem \eqref{eqn:MinimumPbExplicit} and $x \in \partial \{ u>0 \} \cup \partial \{ v>0 \}$. Then:
    \begin{itemize}
        \item[(i)] If $x \in \partial \{ u>0 \} \setminus \partial \{ v>0 \}$ and $u_{\infty}$ is a blow-up limit for $u$ at $x$, of the form
        \[
        u_{\infty}(x) = \Gamma x \cdot \nu
        \]
        for some $\Gamma \in \R$ and $\nu \in S^{N-1}$, then
        \[
        \Gamma = \sqrt{\Lambda_u}.
        \]

        \item[(ii)] If $x \in \partial \{ v>0 \} \setminus \partial \{ u>0 \}$ and $v_{\infty}$ is a blow-up limit for $v$ at $x$, of the form
        \[
        v_{\infty}(x) = \Gamma x \cdot \nu
        \]
        for some $\Gamma \in \R$ and $\nu \in S^{N-1}$, then
        \[
        \Gamma = \sqrt{\Lambda_v}.
        \]

        \item[(iii)] If $x \in \partial \{ u>0 \} \cap \partial \{ v>0 \}$ and $u_{\infty}$, $v_{\infty}$ are blow-up limits for $u$ and $v$ at $x$ respectively, of the form
        \begin{equation}\label{eqn:BlowUpCoincidencePlaneHp}
        u_{\infty}(x) = \Gamma_u x \cdot \nu \quad \text{and} \quad v_{\infty}(x) = \Gamma_u x \cdot \nu
        \end{equation}
        for some $\Gamma_u$, $\Gamma_v \in \R$ and $\nu \in S^{N-1}$, then
        \begin{equation}\label{eqn:BlowUpCoincidencePlaneSum}
        \Gamma_u^2 + \Gamma_v^2 = 1 . 
        \end{equation}
    \end{itemize}
\end{lemma}
\begin{proof}
    Points (i) and (ii) follow exactly as in the one-phase case, see e.g. \cite[Lemma 6.11]{Velichkov:RegularityOnePhaseFreeBd}. Concerning point (iii), we notice that under condition \eqref{eqn:BlowUpCoincidencePlaneHp}, problem \eqref{eqn:BlowUpCoincidencePb} implies that (choosing $\varphi = \Gamma_u x_N$ and $\psi = \Gamma_v x_N$)
    \begin{align*}
    & \int_{B_R} \left\vert \nabla x_N \right\vert^2 + \frac{1}{\Gamma_u^2 + \Gamma_v^2} \left\vert \left\{ x_N>0 \right\} \cap B_R \right\vert \le \\
    & \le \int_{B_R} \left\vert \nabla \left( x_N + \varphi \right) \right\vert^2 + \frac{1}{\Gamma_u^2 + \Gamma_v^2} \left\vert \left( x_N + \varphi>0 \right\} \cap B_R \right\vert,  
    \end{align*}
    for all $R >0$ and $\varphi \in H_0^1(B_R)$. Hence \eqref{eqn:BlowUpCoincidencePlaneSum} follows again from e.g. \cite[Lemma 6.11]{Velichkov:RegularityOnePhaseFreeBd}.
\end{proof}
In order to give an estimate for the dimension of the singular set, as described in Theorem \ref{thm:C1aRegularity}, in the next lemma we show the proportionality of the blow-up limits at the two-phase points.
\begin{lemma}\label{lemma:TwoPhaseBlowUpProportional}
    Let the functions $u$ and $v$ be solutions of problem \eqref{eqn:MinimumPbExplicit} and $x \in \partial \{ u>0 \} \cap \partial \{ v>0 \}$. Suppose that $u_{\infty}$ and $v_{\infty}$ are blow-up limits for $u$ and $v$ at $x$, respectively. 

    Then, there exists a positive constant $c \in \R$ such that %
    \[
    u_{\infty} = c v_{\infty}  \quad \text{in } B_R, \text{ for any } R >0 .
    \]
\end{lemma}
\begin{proof}
    Without loss of generality, we give the proof in the case $R=1$.
    
    By Lemmas \ref{lemma:BlowUpProperties}, \ref{lemma:SolutionsNonDegeneracy} and Corollary \ref{crl:localLipschitzReg}, the functions $u_{\infty}$, $v_{\infty} \in L_{\loc}^{\infty}(\R^N) \cap H_{\loc}^1(\R^N)$ are non-negative, non-vanishing (Lipschitz) continuous functions which are harmonic on their positivity sets. Moreover, using the monotonicity of Wiess' boundary adjusted energy (see e.g. \cite{Weiss99:PartialRegularityFreeBd} or \cite[Section 9]{Velichkov:RegularityOnePhaseFreeBd}), one can show that $u_{\infty}$ and $v_{\infty}$ are one-homogeneous functions in $\R^N$.

    Let $u_{\theta} \in H_0^1\left( \Omega_u \cap S^{N-1} \right)$ and $v_{\theta} \in H_0^1\left( \Omega_u \cap S^{N-1} \right)$ denote the traces of $u_{\infty}$ and $v_{\infty}$ on $\partial B_R$. They solve
    \begin{align}\label{eqn:ProportionalBlowUpsEqns}
    \begin{split}
        & \Delta_S u_{\theta} + (d-1) u_{\theta} = 0 \quad \text{in } \Omega_u \cap S^{N-1} , \\
        & \Delta_S v_{\theta} + (d-1) v_{\theta} = 0 \quad \text{in } \Omega_v \cap S^{N-1} .
    \end{split}
    \end{align}
    Now we proceed ad in \cite[Lemma 2.2]{DePhilippisSpolaorVelichkov2021:TwoPhaseBernoulli}. Let $c \ge 0$ be such that
    \[
    \int_{S^{N-1}} u_{\theta} - c v_{\theta} = 0 .
    \]
    Then, using \eqref{eqn:ProportionalBlowUpsEqns} and the fact that $\Omega_v \subseteq \Omega_u$
    \[
    \int_{S^{N-1}} \left\vert \nabla_{S^{N-1}} (u_{\theta} - c v_{\theta})  \right\vert^2 = (d-1) \int_{S^{N-1}} \vert u_{\theta} - c v_{\theta} \vert^2 . 
    \]
    Hence, the function $u_{\theta} - c v_{\theta}$ is an eigenfunction of the Laplace-Beltrami operator associated to the eigenvalue $d-1$, so that it is a linear function.

    This implies that the positivity set of $u_{\infty}$ contains a half-space, hence it has the inner ball condition and (e.g. by \cite[Lemma 11.17]{CaffarelliSalsa:GeomApproachToFreeBoundary}) it has to coincide with such half-space. But then, also the positivity set of $v_{\infty}$ needs to coincide with the same half-space (since it has the outer ball condition, e.g. again by \cite[Lemma 11.17]{CaffarelliSalsa:GeomApproachToFreeBoundary}).

    Hence, the the function $u_{\theta} - c v_{\theta}$ can coincide with a linear function only if it vanishes identically.
\end{proof}

\subsection{Viscosity setting}\label{subsection:ViscositySetting}
Let us recall the definition of sub/super and viscosity solution for problem \eqref{eqn:MainViscousPb}.
\begin{definition}
    Let $\varphi(x)$, $\psi(x) \in C^0_{loc}(B_1)$. We say that:
    \begin{itemize}
        \item[(i)] $\varphi$ touches $\psi$ from below at $x \in B_1$ if $\varphi(x) = \psi(x)$ and $\varphi \le \psi$ in a neighborhood of $x$,

        \item[(ii)] $\varphi$ touches $\psi$ from above at $x \in B_1$ if $\varphi(x) = \psi(x)$ and $\varphi \ge \psi$ in a neighborhood of $x$.
    \end{itemize}
\end{definition}
\begin{definition}\label{def:viscoudSolOneSidedTwoPhasePb}
    Let $u$, $v \in C^0_{loc}\left( B_1 \right)$. 
    \begin{itemize}
        \item[(i)] We say that the couple $(u, v)$ is a viscosity supersolution of problem \eqref{eqn:MainViscousPb} if
        \begin{itemize}
            \item[1.] $\varphi \in C_{loc}^2(B_1)$ with $\Delta \varphi > 0$ in $B_1$ and $\vert \nabla \varphi \vert^2 > \Lambda_u$ on $\partial \{ \varphi>0 \} \cap B_1$ cannot touch $u$ from below in $\{ u>0 \} \cup \left( \partial \{ u>0 \} \setminus \partial \{ v>0 \} \right)$. We call such $\varphi$ a strict comparison subsolution for $u$.

            \item[2.] $\varphi \in C_{loc}^2(B_1)$ with $\Delta \varphi > 0$ in $B_1$ and $\vert \nabla \varphi \vert^2 > \Lambda_v$ on $\partial \{ \varphi>0 \} \cap B_1$ cannot touch $v$ from below in $\{ v>0 \} \cup \left( \partial \{ v>0 \} \setminus \partial \{ u>0 \} \right)$. We call such $\varphi$ a strict comparison subsolution for $v$.

            \item[3.] $\varphi \in C_{loc}^2(B_1)$ with $\Delta \varphi > 0$ in $B_1$ and $\vert \nabla \varphi \vert^2 > 1$ on $\partial \{ \varphi>0 \} \cap B_1$ cannot touch the average $q_{\eta} \coloneqq 
            ( u, v) \cdot \eta$ (for any $\eta \in S^{N-1}$) from below in $\{ q_{\eta}>0 \} \cup \left( \partial \{ v>0 \} \cap \partial \{ u>0 \} \right)$. We call such $\varphi$ a strict comparison subsolution for $q_{\eta}$.
        \end{itemize}

        \item[(ii)] We say that $(u, v)$ is a viscosity subsolution of problem \eqref{eqn:MainViscousPb} if
        \begin{itemize}
            \item[1.] $\varphi \in C_{loc}^2(B_1)$ with $\Delta \varphi < 0$ in $B_1$ and $\vert \nabla \varphi \vert^2 < \Lambda_u$ on $\partial \{ \varphi>0 \} \cap B_1$ cannot touch $u$ from above in $\{ u>0 \} \cup \left( \partial \{ u>0 \} \setminus \partial \{ v>0 \} \right)$. We call such $\varphi$ a strict comparison supersolution for $u$.

            \item[2.] $\varphi \in C_{loc}^2(B_1)$ with $\Delta \varphi < 0$ in $B_1$ and $\vert \nabla \varphi \vert^2 < \Lambda_v$ on $\partial \{ \varphi>0 \} \cap B_1$ cannot touch $v$ from above in $\{ v>0 \} \cup \left( \partial \{ v>0 \} \setminus \partial \{ u>0 \} \right)$. We call such $\varphi$ a strict comparison supersolution for $v$.

            \item[3.] $\varphi \in C_{loc}^2(B_1)$ with $\Delta \varphi < 0$ in $B_1$ and $\vert \nabla \varphi \vert^2 < 1$ on $\partial \{ \varphi>0 \} \cap B_1$ cannot touch the modulus $m \coloneqq \sqrt{ u^2 + v^2 }$ from above in $\{ m>0 \} \cup \left( \partial \{ v>0 \} \cap \partial \{ u>0 \} \right)$. We call such $\varphi$ a strict comparison supersolution for the modulus $m$.
        \end{itemize}

        \item[(iii)] We say that $(u, v)$ is a viscosity solution of problem \eqref{eqn:MainViscousPb} if it is both a subsolution and a supersolution.
    \end{itemize}
\end{definition}
We are ready to state the following
\begin{lemma}\label{lemma:VariationalSolIsViscousSol}
    Let $(u, v)$ be a solution of \eqref{eqn:MinimumPbExplicit}. Then it is also a viscosity solution for problem \eqref{eqn:MainViscousPb} in the sense of Definition \ref{def:viscoudSolOneSidedTwoPhasePb}.
\end{lemma}
\begin{proof}
The thesis follows by standard elliptic regularity theory at points in the interior of the positivity sets, while from Lemma \ref{lemma:flatBlowUpsRightGradient} and \cite[Lemma 11.17]{CaffarelliSalsa:GeomApproachToFreeBoundary} at the free boundary points.    
\end{proof}

\section{Almost-optimal regularity}\label{section:Regularity}
In this section we prove Theorem \ref{thm:C1aRegularity}. We basically adapt the viscous linearization approach of \cite{DeSilva:FreeBdRegularityOnePhase}, but we also need to take care of the contact points $\partial\{ u>0 \} \cap \partial\{ v>0 \}$. To this aim, we combine ideas for the regularity of the one phase (\cite{AltCaffarelli:OnePhaseFreeBd, DeSilva:FreeBdRegularityOnePhase}), vectorial (\cite{MazzoleniTerraciniVelichkov2017:RegularitySpectralFunctionals, MazzoleniTerraciniVelichkov2020:RegularityVectorialBernoulli, DeSilvaTortone2020:ViscousVectorialBernoulli}) and two phase (\cite{AltCaffarelliFriedman1984:TwoPhaseBernoulli, DePhilippisSpolaorVelichkov2021:TwoPhaseBernoulli}) Bernoulli problems. We proceed as follows.

Following \cite{DeSilva:FreeBdRegularityOnePhase}, the $C^{1, \alpha}$ regularity of the free boundary is a direct consequence of an improvement of flatness result (see Section \ref{subsection:ImprovFlatness}, Theorem \ref{thm:ImprovFlat}), which in turn is proved via a partial Harnack approach (see Section \ref{subsection:BoundaryHarnack}).

Briefly, in order to get Harnack's inequality, we use a combination of the competitors from \cite{DeSilva:FreeBdRegularityOnePhase} for the one phase problem with those from \cite{DeSilvaTortone2020:ViscousVectorialBernoulli} for the vectorial problem, and inspired by \cite{DePhilippisSpolaorVelichkov2021:TwoPhaseBernoulli} we discern the branching points from those where the free boundaries are attached/detached, by looking at the value of the gradient of the blow-up limits.

\subsection{Partial Harnack inequality}\label{subsection:BoundaryHarnack}
In this section we prove a partial Harnack inequality for viscosity solutions to problem \eqref{eqn:MainViscousPb}. To this aim, inspired by \cite{DePhilippisSpolaorVelichkov2021:TwoPhaseBernoulli} concerning the two phase Bernoulli problem, we distinguish two regimes. Consequently, we split the partial Harnack inequality into two main lemmas. The branching regime (namely \eqref{eqn:HarnackIneqBranchingFullHp}) is thought for neighborhoods of points near the branching of the free boundaries, while the coincidence regime (namely \eqref{eqn:HarnackIneqCoincidenceHp}) for neighborhoods where the coincidence of the free boundaries occurs.

We begin studying the branching regime. There are two main  steps in this case. The first one is that we need to deal with both one-phase and two-phase points, while the second one is that we need to take into account for the relative distance between the free boundaries. For this reason, we split the proof of the partial Harnack inequality in two lemmas.

More precisely, Lemma \ref{lemma:PartialHarnackIneqBranching} deals with both one-phase and two-phase points, but the relative distance between the free boundaries is assumed to be virtually zero. On the other hand, Lemma \ref{lemma:PartialHarnackIneqBranchingFull} extends Lemma \ref{lemma:PartialHarnackIneqBranching} also to the case of positive distance, and is basically obtained as an interpolation between Lemma \ref{lemma:PartialHarnackIneqBranchingFull} and \cite[Lemma 3.3]{DeSilva:FreeBdRegularityOnePhase}.
\begin{lemma}\label{lemma:PartialHarnackIneqBranching}
Let $(u, v)$ be a viscosity solution of problem \eqref{eqn:MainViscousPb}. There exist constants $\varepsilon_0, \rho_0, c > 0$ (independent of the specific solution and $\varepsilon$) such that, if
\begin{align} \label{eqn:HarnackIneqBranchingHp}
    \begin{split}
        & \sqrt{\Lambda_u} \left( x_N + \sigma - \varepsilon \right)^+ \le u \le \sqrt{\Lambda_u} \left( x_N + \sigma + \varepsilon\right)^+ \\
        & \sqrt{\Lambda_v} \left( x_N + \sigma - \varepsilon\right)^+ \le v \le \sqrt{\Lambda_v} \left( x_N + \sigma + \varepsilon\right)^+
    \end{split}
\end{align}
for some $\varepsilon < \varepsilon_0$ and $\vert \sigma \vert \le 1/10$, at least one of the following holds:
\begin{align}
    u \ge \sqrt{\Lambda_u} \left( x_N + \sigma - c \varepsilon \right)^+ \text{ and } v \ge \sqrt{\Lambda_v} \left( x_N + \sigma - c \varepsilon \right)^+  \quad \text{in } B_{\rho} , \label{eqn:HarnackImproveBelowBranching}\\ 
    u \le \sqrt{\Lambda_u} \left( x_N + \sigma + (1-c) \varepsilon \right)^+ \text{ and } v \le \sqrt{\Lambda_v} \left( x_N + \sigma + (1-c) \varepsilon \right)^+ \quad \text{in } B_{\rho} \label{eqn:HarnackImproveAboveBranching} , \\
    u \ge \sqrt{\Lambda_u} \left( x_N + \sigma - c \varepsilon \right)^+ \text{ and } v \le \sqrt{\Lambda_v} \left( x_N + \sigma + (1-c) \varepsilon \right)^+ \quad \text{in } B_{\rho} \label{eqn:HarnackIncreaseDistanceBranching}.
\end{align}
\end{lemma}
\begin{proof}
The idea is to consider the competitors for both the one phase \cite{DeSilva:FreeBdRegularityOnePhase} and the vectorial Bernoulli problem \cite{DeSilvaTortone2020:ViscousVectorialBernoulli}, and move them together, to take into account contemporarily the disjoint free boundaries and their intersections.

To simplify the notations, we give the proof only in the case $\sigma = 0$. The case $\sigma \neq 0$ can be dealt with in the same way.

Following \cite{DeSilvaTortone2020:ViscousVectorialBernoulli}, let us introduce the competitors
\begin{equation}\label{eqn:vectorialBernoulliCompetitors}
m \coloneqq \sqrt{u^2 + v^2} \quad \text{and} \quad q \coloneqq \sqrt{\Lambda_u} u + \sqrt{\Lambda_v} v.
\end{equation}
Notice that
\begin{align}\label{eqn:VectorialCompetitorModulusProperties}
\begin{split}
     \{ m>0 \} = \{ u>0 \}, \qquad \Delta m \ge 0 \text{ on } \{ m>0 \}, \\
     \vert \nabla m \vert^2 = \Lambda_u + \Lambda_v = 1 \text{ on } \partial\{ u> 0 \} \cap \partial\{ v> 0 \},
\end{split}
\end{align}
\begin{align}\label{eqn:VectorialCompetitorAverageProperties}
\begin{split}
     \{ q>0 \} = \{ u>0 \}, \qquad \Delta q = 0 \text{ on } \{ q>0 \}, \\
     \vert \nabla q \vert^2 \le \vert \nabla m \vert^2 = 1 \text{ on } \partial\{ u>0 \} \cap \partial\{ v>0 \}.
\end{split}
\end{align}
In particular, no viscous strict comparison supersolution for problem \eqref{eqn:MainViscousPb} can touch $m$ from above on its positivity set or on $\partial\{ u> 0 \} \cap \partial\{ v> 0 \}$. Similarly, no strict comparison subsolution for \eqref{eqn:MainViscousPb} can touch $q$ from below on its positivity set or on $\partial\{ u> 0 \} \cap \partial\{ v> 0 \}$.

For later use, let us also introduce the point $\overline{x} \coloneqq 1/5 e_N$, the set $A \coloneqq B_{3/4}\left( \overline{x} \right) \setminus B_{1/20}\left( \overline{x} \right)$, the functions 
\begin{equation}\label{eqn:partialHarnackCompetitors}
p_u(x) = \sqrt{\Lambda_u} \left( x_N - \varepsilon \right), \quad p_u(x) = \sqrt{\Lambda_u} \left( x_N - \varepsilon \right) , \quad p(x) = x_N - \varepsilon
\end{equation}
and
\begin{equation}\label{eqn:partialHarnackFunctionW}
w(x) \coloneqq
\begin{cases}
    1 & \text{on } \overline{B_{1/20}\left( \overline{x} \right)} , \\
    C \left( \left( x - \overline{x} \right)^{-\gamma} - (3/4)^{-\gamma} \right) & \text{on } A , \\
    0 & \text{on } B_1 \setminus B_{3/4}\left( \overline{x} \right) ,
\end{cases}
\end{equation}
where $C, \gamma > 0$ are chosen so that $w$ is continuous on $\partial B_{1/20}\left( \overline{x} \right)$ and $\Delta w >  0$ on $A$.

Now we distinguish several cases, each of which gives (at least) one among \eqref{eqn:HarnackImproveBelowBranching}, \eqref{eqn:HarnackImproveAboveBranching} and \eqref{eqn:HarnackIncreaseDistanceBranching}.
\begin{itemize}
    \item[(i)] Suppose that
    \begin{equation}\label{eqn:partialHarnackDetachCase1Hp}
    u\left( \overline{x} \right) \le p_u\left( \overline{x} \right) + \sqrt{\Lambda_u} \varepsilon \quad \text{and} \quad v\left( \overline{x} \right) \le p_v\left( \overline{x} \right) + \sqrt{\Lambda_v} \varepsilon .
    \end{equation}
    Notice that in such case we also have
    \[
    q\left( \overline{x} \right) \le p\left( \overline{x} \right) + \varepsilon .
    \]
    Hence, by Harnack's inequality, in $B_{1/20}\left( \overline{x} \right)$
    \begin{equation}\label{eqn:HarnackCase1BoundB1/20}
    p + 2 \varepsilon - q \ge c_0 \varepsilon \quad \text{and} \quad p_u + 2 \sqrt{\Lambda_u} \varepsilon - u \ge c_0 \sqrt{\Lambda_u} \varepsilon
    \end{equation}
    for some $c_0 > 0$ sufficiently small. For $t \in [0, 1]$ let us introduce the competitors
    \begin{align*}
        u_t \coloneqq p_u + & 2 \sqrt{\Lambda_u} \varepsilon + c_0 \sqrt{\Lambda_u} \varepsilon (1 - w) - t c_0 \sqrt{\Lambda_u} \varepsilon , \\
        & m_t \coloneqq p + 2 \varepsilon + c_0 \varepsilon (1 - w) - tc_0 \varepsilon .
    \end{align*}
    The key point that makes the argument work is that the competitors $u_t$ and $m_t$ have the same free boundary in $B_1$, since they are proportional (despite that, we keep different notations for the sake of clarity).
    
    Our aim is to show that
    \begin{equation}\label{eqn:HarnackCase1ToShow}
        u_t^+ \ge  u \quad \text{and} \quad m_t^+ \ge m \quad \text{in } B_{1/2}, \quad \text{for all } t \in [0, 1].
    \end{equation}
    Indeed, if \eqref{eqn:HarnackCase1ToShow} holds, then for $t = 1$
    \begin{equation}\label{eqn:HarnackCase1HalfThesis}
    u \le \left( p_u + 2 \sqrt{\Lambda_u} \varepsilon - c_0 \sqrt{\Lambda_u} w \varepsilon \right)^+ \le \sqrt{\Lambda_u} \left( x_N  + (1 - c) \varepsilon \right)^+ ,
    \end{equation}
    which is the first part of \eqref{eqn:HarnackImproveAboveBranching}. To obtain the other half of the statement, just consider the competitor $v_t$ for $v$ defined by
    \[
    v_t \coloneqq p_v + 2 \sqrt{\Lambda_v} \varepsilon + c_0 \sqrt{\Lambda_v} \varepsilon (1 - w) - t \sqrt{\Lambda_v} c_0 \varepsilon
    \]
    and notice that, thanks to \eqref{eqn:HarnackCase1HalfThesis} it holds $v_t^+ \ge v$ for all $t \in [0, 1]$. Indeed, $v_t$ and $u_t$ are proportional, so that $v_t^+$ cannot touch the free boundary of $v$ unless $t = 1$. Hence we also have
    \[
    v \le \left( p_v + 2 \sqrt{\Lambda_v} \varepsilon - c_0 \sqrt{\Lambda_v} w \varepsilon \right)^+ \le \sqrt{\Lambda_v} \left( x_N  + (1 - c) \varepsilon \right)^+ ,
    \]
    which would complete \eqref{eqn:HarnackImproveAboveBranching}. Thus, we are left to show \eqref{eqn:HarnackCase1ToShow}.

    This is quite standard. For $t = 0$ the assertion comes from the hypotheses \eqref{eqn:HarnackIneqBranchingHp}. By contradiction, let $\overline{t} > 0$ be the smallest $t \in [0, 1)$ where the touching happens for $u_t$. By construction, $u_t > u$ on $\partial B_{3/4} \left( \overline{x} \right)$ and $B_{1/20}\left( \overline{x} \right)$ for all $t \in [0, 1]$, by \eqref{eqn:HarnackCase1BoundB1/20}. Moreover, since this competitor is (by construction) a strict comparison viscous supersolution for problem \eqref{eqn:MainViscousPb}, it can only touch $u$ at $x \in \partial\{ v>0 \} \cap \partial\{ u>0 \}$. However, this would imply that $m_t$ touches $m$ from above at the same point $x$, which, as already recalled, is not possible. Hence \eqref{eqn:HarnackCase1ToShow} is proved.
    
    \item[(ii)] Now suppose
    \begin{equation}\label{eqn:partialHarnackDetachCase2Hp}
    u\left( \overline{x} \right) \ge p_u\left( \overline{x} \right) + \varepsilon/2 \quad \text{and} \quad v\left( \overline{x} \right) \ge p_v\left( \overline{x} \right) + \varepsilon/2 .
    \end{equation}
    This case is similar to the previous one. The only different thing is that one first needs to consider the competitors
    \begin{align*}
        v_t & \coloneqq p_v - c_0 \sqrt{\Lambda_v} \varepsilon (1 - w) + t \sqrt{\Lambda_v} c_0 \varepsilon, \\
        q_t & \coloneqq p - c_0 \varepsilon (1 - w) + t c_0 \varepsilon ,
    \end{align*}
    for $v$ and $q$ respectively, and then the competitor
    \[
    u_t \coloneqq p_u - c_0 \sqrt{\Lambda_u} \varepsilon (1 - w) + t \sqrt{\Lambda_u} c_0 \varepsilon,
    \]
    for $u$. In such case \eqref{eqn:HarnackImproveBelowBranching} is obtained. We omit the details.

    \item[(iii)] In the case
    \begin{equation}\label{eqn:partialHarnackDetachCase3Hp}
    u\left( \overline{x} \right) \le p_u\left( \overline{x} \right) + \varepsilon/2 \quad \text{and} \quad v\left( \overline{x} \right) \ge p_v\left( \overline{x} \right) + \varepsilon/2 ,
    \end{equation}
    there is no a priori direction to gain space, and it depends on the value of $q\left( \overline{x} \right)$. In particular, if
    \[
    q\left( \overline{x} \right) \le p\left( \overline{x} \right) + \varepsilon
    \]
    we proceed similarly as in case (i) to obtain \eqref{eqn:HarnackImproveAboveBranching}, while if
    \[
    q\left( \overline{x} \right) \ge p\left( \overline{x} \right) + \varepsilon
    \]
    we proceed similarly as in case (ii) to deduce \eqref{eqn:HarnackImproveBelowBranching}. We detail only the last case since the first one can be dealt with similarly.

    Suppose $q\left( \overline{x} \right) \ge p\left( \overline{x} \right) + \varepsilon$. Exactly as in case (ii) we can define the competitors
        \begin{align*}
        v_t & \coloneqq p_v - c_0 \sqrt{\Lambda_v} \varepsilon (1 - w) + t \sqrt{\Lambda_v} c_0 \varepsilon, \\
        q_t & \coloneqq p - c_0 \varepsilon (1 - w) + t c_0 \varepsilon ,
    \end{align*}
    and prove that
    \[
     v \ge \sqrt{\Lambda_v} \left( x_N - c \varepsilon \right)^+ \quad \text{and} \quad q \ge \left( x_N - c \varepsilon \right)^+ \quad \text{in } B_{1/2}.
    \]
    Such estimate allows to improve the bound from below on the position of the free boundary not only of $v$, but also of $u$. Now we use such information to prove that
    \begin{equation}\label{eqn:HarnackCase3COmpleteBdBelow}
    u \ge \sqrt{\Lambda_u} \left( x_N - c\sigma\varepsilon \right)^+ \quad \text{in } B_{1/2}.
    \end{equation}
    for some $\sigma>0$ sufficiently small, but independent of $u, v$ and $\varepsilon$. This is actually the main difference with respect to case (ii), since we no longer have the bound $u\left( \overline{x} \right) \ge p_u\left( \overline{x} \right) + \varepsilon/2$. Nontheless, we can proceed as follows.

    Consider a function $\phi : \R^{N-1} \to \R$ with the following properties:
    \begin{align*}
    1 \ge \phi \ge 0 & \text{ on } \R^{N-1}, \qquad
    \phi \in C^{\infty}\left( \R^{N-1} \right), \qquad \phi > 0 \text{ on } B_{3/8} \subset \R^{N-1} , \\
    & \phi = 0 \text{ on } \R^{N-1} \setminus B_{3/8}, \qquad \phi = 1 \text{ on } B_{1/4} \subset \R^{N-1},
    \end{align*}
    and define the diffeomorphism $\Psi: \R^N \to \R^N$ as
    \[
    \Psi(x_1, ..., x_N) = (x_1, ..., x_N - c \varepsilon \phi(x_1, ..., x_{N-1})).
    \]
    Let $T_{\varepsilon} \coloneqq \{ (x_N - \varepsilon)^+ > 0 \} \subset \R^{N}$ and consider the set $S \coloneqq \Psi(T_{\varepsilon})$. Introduce the harmonic function $g$ solving
    \[
    \begin{cases}
        \Delta g = 0 & \text{on } S, \\
        g = p_u = \sqrt{\Lambda_n}(x_N - \varepsilon)^+ & \textbf{on } \partial S \cap \partial T_{\varepsilon}.
    \end{cases}
    \]
    It is possible to prove, for instance by contradiction, that
    \begin{equation}\label{eqn:HarnackCase3UnifBound}
    \text{for any } K \Subset T_{\varepsilon} \text{ there exists a constant } \eta_K > 0 \text{ such that } g - p_u \ge \eta_K \varepsilon \text{ on } K.
    \end{equation}
    We remark that $\eta_K$ is independent of $u, v$ and $\varepsilon$. Suppose for a moment that \eqref{eqn:HarnackCase3UnifBound} holds. Then,
    \[
    u\left( \overline{x} \right) \ge p_u\left( \overline{x} \right) + \eta_K c_0 \varepsilon
    \]
    for some fixed $K$ so that, defining the competitor
    \[
    u_t \coloneqq p_u - \eta_k c_0 \sqrt{\Lambda_u} \varepsilon (1 - w) + t \eta_k c_0 \sqrt{\Lambda_u} \varepsilon
    \]
    for $u$, by contradiction (similarly as in case (ii)) one can prove that $u \ge u_t$ for all $t \in [0, 1]$ in $B_{1/2}$, which ultimately leads to \eqref{eqn:HarnackCase3COmpleteBdBelow}.
    
    We are only left to prove \eqref{eqn:HarnackCase3UnifBound}. Suppose by contradiction that the statement is false. Then, by standard elliptic regularity one can prove that
    \[
    \left \| \frac{g - \sqrt{\Lambda_n}(x_N - \varepsilon)}{\varepsilon} \right \|_{C^{0, \alpha}\left( \overline{T_{\varepsilon}} \right)} \le C
    \]
    for some $0 < \alpha < 1$ and $C>0$ uniformly in $\varepsilon$. Thus, up to a subsequence, there exists a limit function $z \in C^{0, \alpha}\left( \overline{T_{\varepsilon}} \right)$ which is harmonic in $T_0 = \{ x_N > 0 \} \cap B_1$. Moreover, by the contradiction assumption and Harnack's inequality, $z = 0$ on $T_0$. On the other hand, by the strong convergence and the shape of $S = \Psi(T_{\varepsilon})$ it holds $z > 0$ on $\{ x_N = 0 \} \cap \partial T_0$ (by the choice of the boundary data for g). Thus we have reached a contradiction and \eqref{eqn:HarnackCase3UnifBound} is proved.
    
    \item[(iv)] The last case is
    \begin{equation}\label{eqn:partialHarnackDetachCase4Hp}
    u\left( \overline{x} \right) \ge p_u\left( \overline{x} \right) + \varepsilon/2 \quad \text{and} \quad v\left( \overline{x} \right) \le p_v\left( \overline{x} \right) + \varepsilon/2 .
    \end{equation}
    This is actually the simplest one and can be carried out exactly as in \cite{DeSilva:FreeBdRegularityOnePhase} for the one phase Bernoulli problem. Indeed, considering the competitors
    \begin{align*}
        u_t & \coloneqq p_u - c_0 \sqrt{\Lambda_u} \varepsilon (1 - w) + t \sqrt{\Lambda_u} c_0 \varepsilon, \\
        v_t & \coloneqq p_v + 2 \sqrt{\Lambda_v} \varepsilon + c_0 \sqrt{\Lambda_v} \varepsilon (1 - w) - t \sqrt{\Lambda_v} c_0 \varepsilon ,
    \end{align*}
    For $u$ and $v$ respectively, now it is not a problem if they touch points in $\partial\{ v>0 \} \cap \partial\{ u>0 \}$ (actually this time it strengthens the contradiction). In this case we get \eqref{eqn:HarnackIncreaseDistanceBranching}.
    
\end{itemize}
\end{proof}
\begin{lemma}\label{lemma:PartialHarnackIneqBranchingFull}
Let $(u, v)$ be a viscosity solution of problem \eqref{eqn:MainViscousPb}. There exist constants $\varepsilon_0, \rho_0, c > 0$ (independent of the specific solution and $\varepsilon$) such that, if
\begin{align} \label{eqn:HarnackIneqBranchingFullHp}
    \begin{split}
        & \sqrt{\Lambda_u} \left( x_N + \sigma_u \right)^+ \le u \le \sqrt{\Lambda_u} \left( x_N + \sigma_u + \varepsilon\right)^+ , \\
        & \sqrt{\Lambda_v} \left( x_N + \sigma_v \right)^+ \le v \le \sqrt{\Lambda_v} \left( x_N + \sigma_v + \varepsilon\right)^+ , \\
        & \vert \sigma_u \vert, \vert \sigma_v \vert \le 1/10 \quad \text{and} \quad \sigma_u \ge \sigma_v
    \end{split}
\end{align}
for some $\varepsilon < \varepsilon_0$, then for some $\sigma'_u, \sigma'_v \in \R$ we have:
\begin{align}
    & \sigma'_u \ge \sigma'_v , \label{eqn:HarnackBranchingFullSigmaPrimeOrdered} \\
    \sqrt{\Lambda_u} \left( x_N + \sigma_u \right)^+ & \le u \le \sqrt{\Lambda_u} \left( x_N + \sigma_u + (1-c) \varepsilon\right)^+ , \label{eqn:HarnackImproveUBranchingFull} \\ 
    \sqrt{\Lambda_v} \left( x_N + \sigma_v \right)^+ & \le v \le \sqrt{\Lambda_v} \left( x_N + \sigma_v + (1-c) \varepsilon\right)^+ . \label{eqn:HarnackImproveVBranchingFull}
\end{align}
\end{lemma}
\begin{proof}
The idea of the proof is to consider two cases: if $\left\vert \sigma_u - \sigma_v \right\vert$ is small enough, then the two free boundaries are expected to be moved together, basically thanks to Lemma \ref{lemma:PartialHarnackIneqBranching}. If, on the other hand, $\left\vert \sigma_u - \sigma_v \right\vert$ is large, then the two free boundary can be moved separately, basically treating them as one-phase free boundaries adapting the arguments in \cite{DeSilva:FreeBdRegularityOnePhase}.

To begin with, we prove the validity of \eqref{eqn:HarnackImproveUBranchingFull} and \eqref{eqn:HarnackImproveVBranchingFull} for some $\sigma'_u, \sigma'_v \in \R$, and after this we will prove that one can actually choose $\sigma'_u$ and $\sigma'_v$ so that \eqref{eqn:HarnackBranchingFullSigmaPrimeOrdered} holds.

In order to prove \eqref{eqn:HarnackImproveUBranchingFull} and \eqref{eqn:HarnackImproveVBranchingFull} we distinguish two cases.

\begin{itemize}
    \item[(i)] Let $c>0$ be the constant given in Lemma \ref{lemma:PartialHarnackIneqBranching}, and suppose that
    \begin{equation}\label{eqn:HarnackIneqBranchingFullCase1Hp}
    \left\vert \sigma_u - \sigma_v \right\vert \le \frac{c}{2} \varepsilon .
    \end{equation}
    Under \eqref{eqn:HarnackIneqBranchingFullCase1Hp} and \eqref{eqn:HarnackIneqBranchingFullHp} we have that
   \begin{align*}
    & \sqrt{\Lambda_u} \left( x_N + \sigma_v \right)^+ \le u \le \sqrt{\Lambda_u} \left( x_N + \sigma_v + (1+c/2) \varepsilon\right)^+ , \\ 
    & \sqrt{\Lambda_v} \left( x_N + \sigma_v \right)^+ \le v \le \sqrt{\Lambda_v} \left( x_N + \sigma_v + (1+c/2) \varepsilon\right)^+ ,
   \end{align*}
   A direct application of Lemma \ref{lemma:PartialHarnackIneqBranching} gives
   \begin{align*}
    & \sqrt{\Lambda_u} \left( x_N + \sigma_v \right)^+ \le u \le \sqrt{\Lambda_u} \left( x_N + \sigma_v + (1-c) (1+c/2) \varepsilon\right)^+ , \\ 
    & \sqrt{\Lambda_v} \left( x_N + \sigma_v \right)^+ \le v \le \sqrt{\Lambda_v} \left( x_N + \sigma_v + (1-c)(1+c/2) \varepsilon\right)^+ ,
   \end{align*}
   so that \eqref{eqn:HarnackImproveUBranchingFull} and \eqref{eqn:HarnackImproveVBranchingFull} hold with $\sigma'_u = \sigma'_v = \sigma_v$ and constant $c/4$.

   \item[(ii)] Suppose that
    \begin{equation}\label{eqn:HarnackIneqBranchingFullCase2Hp}
    \left\vert \sigma_u - \sigma_v \right\vert \ge \frac{c}{2} \varepsilon ,
    \end{equation}
    where $c>0$ denotes again the constant given in Lemma \ref{lemma:PartialHarnackIneqBranching}.
    
    Let us introduce the point $\overline{x} \coloneqq 1/5 e_N$ and consider the functions
    \[
      p_u(x) \coloneqq \sqrt{\Lambda_u} \left( x_N + \sigma_u \right) \quad \text{and} \quad p_v (x) \coloneqq \sqrt{\Lambda_v} \left( x_N + \sigma_v \right).
   \]
    We need to distinguish four cases.
    \begin{align}
        & u\left( \overline{x} \right) \le p_u\left( \overline{x} \right) + \varepsilon/2 \quad \text{and} \quad v\left( \overline{x} \right) \le p_v\left( \overline{x} \right) + \varepsilon/2 ,
        \label{eqn:HarnackIneqBranchingFullCase2Case1Hp} \\
        & u\left( \overline{x} \right) \ge p_u\left( \overline{x} \right) + \varepsilon/2 \quad \text{and} \quad v\left( \overline{x} \right) \ge p_v\left( \overline{x} \right) + \varepsilon/2 ,
        \label{eqn:HarnackIneqBranchingFullCase2Case2Hp} \\
        & u\left( \overline{x} \right) \le p_u\left( \overline{x} \right) + \varepsilon/2 \quad \text{and} \quad v\left( \overline{x} \right) \ge p_v\left( \overline{x} \right) + \varepsilon/2 ,
        \label{eqn:HarnackIneqBranchingFullCase2Case3Hp} \\
        & u\left( \overline{x} \right) \ge p_u\left( \overline{x} \right) + \varepsilon/2 \quad \text{and} \quad v\left( \overline{x} \right) \le p_v\left( \overline{x} \right) + \varepsilon/2 .
        \label{eqn:HarnackIneqBranchingFullCase2Case4Hp}
    \end{align}
    Case \eqref{eqn:HarnackIneqBranchingFullCase2Case4Hp} can be dealt with exactly as point (iv) of Lemma \ref{lemma:PartialHarnackIneqBranching}, hence its proof will be omitted. Moreover, cases \eqref{eqn:HarnackIneqBranchingFullCase2Case1Hp}, \eqref{eqn:HarnackIneqBranchingFullCase2Case2Hp} and \eqref{eqn:HarnackIneqBranchingFullCase2Case3Hp} can be dealt with with the same strategy, which is an adaptation of \cite[Lemma 3.3]{DeSilva:FreeBdRegularityOnePhase}. For these reasons, we only sketch the proof of \eqref{eqn:HarnackImproveUBranchingFull} and \eqref{eqn:HarnackImproveVBranchingFull} for case \eqref{eqn:HarnackIneqBranchingFullCase2Case1Hp}, highlighting the main points which allow to run the argument in \cite[Lemma 3.3]{DeSilva:FreeBdRegularityOnePhase}.

    In order to deal with case \eqref{eqn:HarnackIneqBranchingFullCase2Case1Hp} let the function $w$ as in \eqref{eqn:partialHarnackFunctionW} and introduce the competitors
    \begin{align*}
        & u_t \coloneqq p_u + \sqrt{\Lambda_u} \varepsilon + \delta \varepsilon (1-w) - t \delta \varepsilon , \\
        & v_t \coloneqq p_v  - \delta \varepsilon (1-w) + t \delta \varepsilon ,
    \end{align*}
    where the constant $\delta > 0$ is fixed (independently of $u, v, \varepsilon, \sigma_u, \sigma_v$) so that the free boundaries of $u_t$ and $v_t$ lie within $\varepsilon c/2$-neighborhoods of $\{ x_N = - \sigma_u - \varepsilon \} \cap B_{3/4}\left( \overline{x} \right)$ and $\{ x_N = - \sigma_v  \} \cap B_{3/4}\left( \overline{x} \right)$ respectively, for all $t \in [0, 1]$.

    Since $u_t$ and $v_t$ are strict comparison super/subsolutions for $u$ and $v$ respectively (in the sende of Definition \ref{def:viscoudSolOneSidedTwoPhasePb}), and their free boundary cannot touch a two phase point in $B_{3/4}\left( \overline{x} \right)$ for any $t \in [0, 1]$ (by the choice of $\delta$), proceeding exactly as in \cite[Lemma 3.3]{DeSilva:FreeBdRegularityOnePhase} one can derive \eqref{eqn:HarnackImproveUBranchingFull} and \eqref{eqn:HarnackImproveVBranchingFull}.
\end{itemize}
We are only left to prove \eqref{eqn:HarnackBranchingFullSigmaPrimeOrdered}. Moreover, we only need to check it for case (ii), since in case (i) it is trivially verified. To this aim, it is sufficient to choose the constant $\delta$ so small that the constant in \eqref{eqn:HarnackImproveUBranchingFull} and \eqref{eqn:HarnackImproveVBranchingFull} arising in case (ii) is smaller then $c/4$.

\end{proof}
Now we deal with the coincidence regime. The partial Harnack result in this case is contained in the following
\begin{lemma}\label{lemma:PartialHarnackIneqCoincidence}
Let $(u, v)$ be a viscosity solution of problem \eqref{eqn:MainViscousPb}. There exist constants $\varepsilon_0, \rho_0, c, C, \theta > 0$ (independent of the specific solution and $\varepsilon$) such that, if
\begin{align} \label{eqn:HarnackIneqCoincidenceHp}
    \begin{split}
        & \Gamma_u \left( x_N + \sigma - \varepsilon \right)^+ \le u \le \Gamma_u \left( x_N + \sigma + \varepsilon\right)^+ , \\
        & \Gamma_v \left( x_N + \sigma - \varepsilon\right)^+ \le v \le \Gamma_v \left( x_N + \sigma + \varepsilon\right)^+ , \\
        & \left\vert \Gamma_u -  \sqrt{\Lambda_u} \right\vert \ge C \varepsilon \text{ and } \left\vert \Gamma_v -  \sqrt{\Lambda_v} \right\vert \ge C \varepsilon ,
    \end{split}
\end{align}
for some $\varepsilon < \varepsilon_0$ and $\vert \sigma \vert \le 1/10$, then at least one of the following holds:
\begin{align}
    u \ge \Gamma_u \left( x_N + \sigma - c \varepsilon \right)^+ \text{ and } v \ge \Gamma_v \left( x_N + \sigma - c \varepsilon \right)^+  \quad \text{in } B_{\rho} , \label{eqn:HarnackImproveBelowCoincidence}\\ 
    u \le \Gamma_u \left( x_N + \sigma + (1-c) \varepsilon \right)^+ \text{ and } v \le \Gamma_v \left( x_N + \sigma + (1-c) \varepsilon \right)^+ \quad \text{in } B_{\rho} \label{eqn:HarnackImproveAboveCoincidence}.
\end{align}
\end{lemma}
\begin{remark}
Notice that in the coincidence regime \eqref{eqn:HarnackIneqCoincidenceHp}, the partial Harnack inequality allows to move the free boundaries together and in the same direction, thus removing the possibility of detachment as described in \eqref{eqn:HarnackIncreaseDistanceBranching}.
\end{remark}
\begin{remark}\label{rmk:BddGammauvCoincidence}
    In the coincidence regime \eqref{eqn:HarnackIneqCoincidenceHp}, as long as $0 < \sqrt{\Lambda_u}, \sqrt{\Lambda_v} < 1$ strictly, one can choose $\varepsilon_0$ so small that the following bounds hold:
    \begin{align*}
        \begin{split}
            & \sqrt{\Lambda_v}/2 \le \Gamma_u \le \sqrt{\Lambda_u} , \\
            & \sqrt{\Lambda_v} \le \Gamma_v \le 1 .
        \end{split}
    \end{align*}
    The proof of this fact can be carried out by contradiction, following \cite[Lemma 4.1]{FerreriVelichov2023:BernoulliInternalInclusion}. Actually, the bounds from above and below for $\Gamma_u$ and $\Gamma_v$ respectively, are also remarked at the beginning of the proof of Lemma \ref{lemma:PartialHarnackIneqCoincidence}.
\end{remark}
\begin{proof}[Proof of Lemma \ref{lemma:PartialHarnackIneqCoincidence}.]
Analogously as for Lemma \ref{lemma:PartialHarnackIneqBranching}, we only give the proof in the case $\sigma = 0$.

To begin with, we notice that neither of the cases $\Gamma_u \ge \sqrt{\Lambda_u} + C \varepsilon$ or $\Gamma_v \le \sqrt{\Lambda_v} - C \varepsilon$ in \eqref{eqn:HarnackIneqCoincidenceHp} can actually occur. Indeed, let us suppose by contradiction that $\Gamma_u \ge \sqrt{\Lambda_u} + C \varepsilon$ (the other case can be dealt with analogously). Then, for sufficiently large $C$ depending only on $\varepsilon_0$, one can slightly deform the competitor $\Gamma_u \left( x_N - \varepsilon\right)^+$ to a function that touches $u$ from below at a point of the free boundary, but has gradient strictly larger then $\sqrt{\Lambda_u}$ (see for instance \cite[Lemma 4.1]{FerreriVelichov2023:BernoulliInternalInclusion}). A contradiction.

We are left to consider the case
\[
\Gamma_u \le  \sqrt{\Lambda_u} - C \varepsilon \quad \text{and} \quad \Gamma_v \ge  \sqrt{\Lambda_v} + C \varepsilon .
\]
The idea is to proceed similarly as in the vectorial Bernoulli problem \cite{DeSilvaTortone2020:ViscousVectorialBernoulli}. Since this strategy is similar to the one used for points (i), (ii) and (iii) in the proof of Lemma \ref{lemma:PartialHarnackIneqBranching}, we will only sketch the common parts and highlight the main differences.

Let us consider the point $\overline{x} \coloneqq 1/5 e_N$ and the competitors $m(x)$ and $q(x)$ as in \eqref{eqn:vectorialBernoulliCompetitors} but with $\Gamma_u$ and $\Gamma_v$ in place of $\sqrt{\lambda_u}$ and $\sqrt{\Lambda_v}$ respectively, satisfying properties analogous to \eqref{eqn:VectorialCompetitorModulusProperties}, \eqref{eqn:VectorialCompetitorAverageProperties}.

Now distinguish two mutually exclusive cases.
\begin{itemize}
    \item[(i)] Suppose that
    \[
    q\left( \overline{x} \right) \le p\left( \overline{x} \right) + \varepsilon,
    \]
    where $p(x)$ is defined in \eqref{eqn:partialHarnackCompetitors}. Notice that thanks to \eqref{eqn:HarnackIneqCoincidenceHp} we also have
    \begin{align*}
    & u\left( \overline{x} \right) \le \Gamma_u \left( \overline{x_N} + \varepsilon\right)^+ \le \sqrt{\Lambda_u} \left( \overline{x_N} + \varepsilon\right) - C \varepsilon  \left( \overline{x_N} + \varepsilon\right) = \\
    & = \sqrt{\Lambda_u} \overline{x_N} + \varepsilon \left( \sqrt{\Lambda_u} - C \overline{x_N} - C \varepsilon \right) \le \sqrt{\Lambda_u} \overline{x_N} = p_u\left( \overline{x_N} \right) + \varepsilon,
    \end{align*}
    where the last inequality follows e.g. for $C \ge 5\sqrt{\Lambda_u}$.

    Hence, we can proceed similarly as in case (i) of Lemma \ref{lemma:PartialHarnackIneqBranching}, proving \eqref{eqn:HarnackImproveAboveCoincidence}.

    \item[(ii)] Now suppose that
    \[
    q\left( \overline{x} \right) \ge p\left( \overline{x} \right) + \varepsilon.
    \]
    Noticing that
    \begin{align*}
    & v\left( \overline{x} \right) \ge \Gamma_u \left( \overline{x_N} - \varepsilon\right)^+ \ge \sqrt{\Lambda_u} \left( \overline{x_N} - \varepsilon\right) + C \varepsilon  \left( \overline{x_N} - \varepsilon\right) = \\
    & = \sqrt{\Lambda_u} \overline{x_N} + \varepsilon \left( C \overline{x_N} - \sqrt{\Lambda_u} - C \varepsilon \right) \ge \sqrt{\Lambda_u} \overline{x_N} = p_u\left( \overline{x_N} \right) + \varepsilon,
    \end{align*}
    for $C \ge \sqrt{\Lambda_u}/(1/5 - \varepsilon_0)$, we can proceed similarly as in case (ii) of Lemma \ref{lemma:PartialHarnackIneqBranching}, deriving \eqref{eqn:HarnackImproveBelowCoincidence}.
\end{itemize}
\end{proof}

\subsection{Improvement of flatness}\label{subsection:ImprovFlatness}

In this section we prove an improvement of flatness result for problem \eqref{eqn:MainViscousPb}, following the strategy of \cite{DeSilva:FreeBdRegularityOnePhase}. More precisely, our aim is to show the following
\begin{theorem}\label{thm:ImprovFlat}
    Let $(u, v)$ be a viscosity solution of problem \eqref{eqn:MainViscousPb} with $ 0 \in \partial\{ u>0 \} \cap \partial\{ v>0 \}$. Suppose that for some constants $\Gamma_u, \Gamma_v > 0$ and $\nu \in S^{N-1}$
\begin{align} \label{eqn:ImprovFlatnessHp}
    \begin{split}
        \Gamma_u (x \cdot \nu - \varepsilon)^+ & \le u \le \Gamma_u (x \cdot \nu + \varepsilon)^+ \quad \text{in } B_1, \\
        \Gamma_v (x \cdot \nu - \varepsilon)^+ & \le v \le \Gamma_v (x \cdot \nu + \varepsilon)^+ \quad \text{in } B_1, \\
        & \Gamma_u^2 + \Gamma_v^2 = 1.
    \end{split}
\end{align}
    Then, there exist constants $0< \rho, \sigma < 1$, $\varepsilon_0, C>0$ (independent of $u, \varepsilon, \nu, \Gamma_u, \Gamma_v$), a vector $\nu' \in S^{N-1}$ and real numbers $\Gamma_u', \Gamma_v'$ such that if $\varepsilon \le \varepsilon_0$
    \begin{align}
        & \Gamma_u^{'2} + \Gamma_v^{'2} = 1 , \label{eqn:ImprovFlatThmSum1}\\
        & \vert \Gamma_u - \Gamma_u' \vert + \vert \Gamma_v - \Gamma_v' \vert + \vert \nu - \nu' \vert \le C \varepsilon , \label{eqn:ImprofFlatThmErrors}\\
        & \Gamma_u' (x \cdot \nu' - \rho \sigma \varepsilon)^+ \le u \le \Gamma_u' (x \cdot \nu' + \rho \sigma \varepsilon)^+ \quad \text{in } B_{\rho}, \label{eqn:ImprovFlatThmU} \\
        & \Gamma_v' (x \cdot \nu' - \rho \sigma \varepsilon)^+ \le v \le \Gamma_v' (x \cdot \nu' + \rho \sigma \varepsilon)^+ \quad \text{in } B_{\rho} \label{eqn:ImprovFlatThmV}.
    \end{align}
\end{theorem}
The proof is carried out by contradiction and is divided into three main steps. Given sequences $\Gamma_{u, n}, \Gamma_{v, n}>0$, $\varepsilon_n \to 0^+$ and $(u_n, v_n)$ solutions of problem \eqref{eqn:MainViscousPb} satisfying the hypotheses \eqref{eqn:ImprovFlatnessHp}, depending on the regime one can introduce the following functions.
\begin{itemize}
    \item[(i)] If, up to a subsequence, 
    \begin{equation}\label{eqn:ImprovFlatRegimeBranching}
    \left\vert \Gamma_{u, n} - \sqrt{\Lambda_u} \right\vert \le C \varepsilon_n \quad \text{and} \quad  \left\vert \Gamma_{v, n} - \sqrt{\Lambda_v} \right\vert \le C \varepsilon_n,
    \end{equation}
    where $C>0$ is the constant of Lemma \ref{lemma:PartialHarnackIneqCoincidence}, consider
    \begin{equation}\label{eqn:TildeunTildevnDefBranching}
      \tilde{u}_n \coloneqq \frac{u_n - \sqrt{\Lambda_u} x_N}{\sqrt{\Lambda_u} \varepsilon_n} \text{ in } \{ u>0 \} \quad \text{and} \quad \tilde{v}_n \coloneqq \frac{v_n - \sqrt{\Lambda_v} x_N}{\sqrt{\Lambda_v} \varepsilon_n} \text{ in } \{ v>0 \} .
    \end{equation}
    Notice that thanks to \eqref{eqn:ImprovFlatnessHp} and \eqref{eqn:ImprovFlatRegimeBranching} we have
    \[
    -(C+1) \le \tilde{u}_n \le 1 \quad \text{and} \quad -(C+1) \le \tilde{v}_n \le 1.
    \]

    \item[(ii)] If, on the other hand, 
    \begin{equation}\label{eqn:ImprovFlatRegimeCoincidence}
     \left\vert \Gamma_{u, n} - \sqrt{\Lambda_u} \right\vert \ge C \varepsilon_n \quad \text{and} \quad  \left\vert \Gamma_{v, n} - \sqrt{\Lambda_v} \right\vert \ge C \varepsilon_n,
    \end{equation}
    with $C>0$ again the constant from Lemma \ref{lemma:PartialHarnackIneqCoincidence}, consider the functions
    \begin{equation}\label{eqn:TildeunTildevnDefCoincidence}
      \tilde{u}_n \coloneqq \frac{u_n - \Gamma_{u, n} x_N}{\Gamma_{u, n} \varepsilon_n} \text{ in } \{ u_n>0 \} \quad \text{and} \quad \tilde{v}_n \coloneqq \frac{v_n - \Gamma_{v, n} x_N}{\Gamma_{v, n} \varepsilon_n} \text{ in } \{ u_n>0 \} .
    \end{equation}
    Notice that in this case, both $\tilde{u}_n$ and $\tilde{v}_n$ are defined on the positivity set of $u_n$. The linearization of $v_n$ over the domain of positivity of $u_n$ is only possible in the coincidence regime.

    Thanks to \eqref{eqn:ImprovFlatnessHp} and the condition $\{ v_n>0 \} \subset \{ u_n>0 \}$, we have
    \begin{equation}\label{eqn:ImprovFlatTildeMUnifBound}
        -1 \le \tilde{u}_n \le 1 \quad \text{and} \quad -1 \le \tilde{v}_n \le 1.
    \end{equation}
\end{itemize}
Then, the idea is first of all to prove a compactness result for the functions $\tilde{u}_n$ and $\tilde{v}_n$ or $\tilde{m}_n$, which are the content of Lemmas \ref{lemma:compactnessBranching} and \ref{lemma:compactnessCoincidence} below respectively. After this, one can study the regularity of the limit functions, which turn out to be solutions (in the viscous sense) of a limit linearized problem (see Lemmas \ref{lemma:linearizationBranching} and \ref{lemma:linearizationCoincidence}). Then, one can use the regularity of the limit functions to prove Theorem \ref{thm:ImprovFlat} by contradiction, and this is the content of section \ref{subsubsec:improvement}.
\begin{remark}\label{rmk:ImprovFlatEstimateLinearizedFunctions}
    Similarly as in \cite{ChangLaraSavin:BoundaryRegularityOnePhase}, we remark that Theorem \ref{thm:ImprovFlat} implies the existence of universal constants $\varepsilon_0, C>0$ such that if
    \[
    \varepsilon_n \| \tilde{u}_n \|_{L^{\infty}\left( \Omega_u^+ \cap B_1 \right)} \le \varepsilon_0 \quad \text{and} \quad \varepsilon_n \| \tilde{v}_n \|_{L^{\infty}\left( \Omega_v^+ \cap B_1 \right)} \le \varepsilon_0
    \]
    then
    \begin{align*}
        \| \tilde{u}_n \|_{C^{1, \alpha}\left( \overline{\Omega_u^+} \cap B_{1/2} \right)} \le C \| \tilde{u}_n \|_{L^{\infty}\left( \Omega_u^+ \cap B_1 \right)} \quad \text{and} \quad \| \tilde{v}_n \|_{C^{1, \alpha}\left( \overline{\Omega_v^+} \cap B_{1/2} \right)} \le C \| \tilde{v}_n \|_{L^{\infty}\left( \Omega_v^+ \cap B_1 \right)} .
    \end{align*}
    This estimate will be one of the key points in proving the sharp regularity, namely Theorem \ref{thm:SharpRegularity}.
\end{remark}

\subsubsection{Compactness}\label{subsubsec:compactness}
The aim of this section is to prove a compactness result for the linearizing sequences \eqref{eqn:TildeunTildevnDefBranching} and \eqref{eqn:TildeunTildevnDefCoincidence}. More precisely, in the branching regime, our aim is to prove the following
\begin{lemma}\label{lemma:compactnessBranching}
    Let $\varepsilon_n \to 0^+$, $\Gamma_{u, n}, \Gamma_{v, n}>0$, and $(u_n, v_n)$ solutions of problem \eqref{eqn:MainViscousPb} satisfying  \eqref{eqn:ImprovFlatnessHp}. Moreover suppose that \eqref{eqn:ImprovFlatRegimeBranching} holds and let the functions $\tilde{u}_n$, $\tilde{v}_n$ be defined as in \eqref{eqn:TildeunTildevnDefBranching}.

    Then, there exist two functions $\tilde{u}_{\infty}$, $\tilde{v}_{\infty} \in C^{0, \alpha}\left( \{ x_N \ge 0 \} \cap B_{1/2} \right)$ for some $0 < \alpha \le 1$ such that:
    \begin{itemize}
        \item[(i)] $\tilde{u}_n \to \tilde{u}_{\infty}$ and $\tilde{v}_n \to \tilde{v}_{\infty}$ in $C^{0, \alpha}\left( \{ x_N \ge \delta \} \cap B_{1/2} \right)$ for all $\delta > 0$.

        \item[(ii)] In the Hausdorff distance,
        \begin{align*}
            & \Theta_{u, n} \coloneqq \left\{ \left(x, \tilde{u}_n(x)\right) \in \R^{N+1} : x \in \{ u_n>0 \} \right\} \to \Theta_{u, \infty} \coloneqq \left\{ \left(x, \tilde{u}_{\infty}(x)\right) \in \R^{N+1} : x \in \{ x_N\ge0 \} \right\}, \\
            & \Theta_{v, n} \coloneqq \left\{ \left(x, \tilde{v}_n(x)\right) \in \R^{N+1} : x \in \{ v_n>0 \} \right\} \to \Theta_{v, \infty} \coloneqq \left\{ \left(x, \tilde{v}_{\infty}(x)\right) \in \R^{N+1} : x \in \{ x_N\ge0 \} \right\}.
        \end{align*}
    \end{itemize}
\end{lemma}
The (similar) result in the coincidence regime is the following
\begin{lemma}\label{lemma:compactnessCoincidence}
    Let $\varepsilon_n \to 0^+$, $\Gamma_{u, n}, \Gamma_{v, n}>0$, and $(u_n, v_n)$ solutions of problem \eqref{eqn:MainViscousPb} satisfying  \eqref{eqn:ImprovFlatnessHp}. Moreover suppose that \eqref{eqn:ImprovFlatRegimeCoincidence} holds and let the functions $\tilde{u}_n$, $\tilde{v}_n$ be defined as in \eqref{eqn:TildeunTildevnDefCoincidence}.

    Then, there exist functions $\tilde{u}_{\infty}$, $\tilde{v}_{\infty} \in C^{0, \alpha}\left( \{ x_N \ge 0 \} \cap B_{1/2} \right)$ for some $0 < \alpha \le 1$ such that:
    \begin{itemize}
        \item[(i)] $\tilde{u}_n \to \tilde{u}_{\infty}$ and $\tilde{v}_n \to \tilde{v}_{\infty}$ in $C^{0, \alpha}\left( \{ x_N \ge \delta \} \cap B_{1/2} \right)$ for all $\delta > 0$.

        \item[(ii)] $\tilde{u}_{\infty} = \tilde{v}_{\infty}$ on $\{ x_N = 0 \} \cap B_{1/2}$.

        \item[(iii)] In the Hausdorff distance,
        \begin{align*}
            & \Theta_{u, n} \to \Theta_{u, \infty}, \\
            & \Theta_{v, n} \to \Theta_{v, \infty} .
        \end{align*}
    \end{itemize}
\end{lemma}
\begin{remark}\label{rmk:TildeUnUnInftyCoincidenceXn0}
 Notice that in the coincidence regime we have the additional condition $\tilde{u}_{\infty} = \tilde{v}_{\infty}$ on $\{ x_N = 0 \} \cap B_{1/2}$, which in general is false in the branching regime. Such condition is a direct consequence of the fact that $u_n$ and $v_n$ are linearized on the largest domain of positivity, namely that of $u_n$, at whose boundary $\tilde{u}_{n} = \tilde{v}_{n} = x_N/\varepsilon_n$ by definitions \eqref{eqn:TildeunTildevnDefCoincidence}.
\end{remark}
We begin with Lemma \ref{lemma:compactnessBranching}, thus dealing with the branching regime. The proof is based on the following Hölder continuity lemma, which in turn relies on Harnack's inequality (see Lemma \ref{lemma:PartialHarnackIneqBranchingFull}). Once Lemma \ref{lemma:HolderTildeunTildevnBranching} is proved, Lemma \ref{lemma:compactnessBranching} basically follows from an 
 Ascoli-Arzelà compactness argument (see e.g. \cite{DeSilva:FreeBdRegularityOnePhase} or \cite[Lemma 7.15]{Velichkov:RegularityOnePhaseFreeBd}), of which we omit the proof. 
\begin{lemma}\label{lemma:HolderTildeunTildevnBranching}
    Under the hypotheses of Lemma \ref{lemma:compactnessBranching} there exist constants $c>0$ and $0 < \beta \le 1$ such that
    \begin{itemize}
        \item[(i)] for all $\overline{x} \in \overline{\Omega_{u_n}} \cap B_{1/2}$
        \begin{equation}\label{eqn:HolderTildeunBranching}
        \left\vert \tilde{u}_n\left( x \right) - \tilde{u}_n\left( \overline{x} \right) \right\vert \le c \left\vert  x - \overline{x} \right\vert^{\beta} \quad \text{if } x \in \overline{\Omega_{u_n}} \cap \left( B_{1/2}\left( \overline{x} \right) \setminus B_{\varepsilon_n/\varepsilon_0}\left( \overline{x} \right) \right) ,
        \end{equation}
        \item[(ii)] for all $\overline{x} \in \overline{\Omega_{v_n}} \cap B_{1/2}$
        \begin{equation}\label{eqn:HolderTildevnBranching}
        \left\vert \tilde{v}_n\left( x \right) - \tilde{v}_n\left( \overline{x} \right) \right\vert \le c \left\vert  x - \overline{x} \right\vert^{\beta} \quad \text{if } x \in \overline{\Omega_{v_n}} \cap \left( B_{1/2}\left( \overline{x} \right) \setminus B_{\varepsilon_n/\varepsilon_0}\left( \overline{x} \right) \right) .
        \end{equation}
    \end{itemize}
\end{lemma}
\begin{proof}
    The basic idea is to iterate Lemma \ref{lemma:PartialHarnackIneqBranchingFull}.

    Fix $n \in \N$. In the following we denote $u_n = u$. Now let $\overline{n}$ be the largest integer such that $\varepsilon_n \le \rho^{\overline{n}} \varepsilon_0$, where $\rho$ and $\varepsilon_0$ are as in Lemma \ref{lemma:PartialHarnackIneqBranchingFull}. Hence, $\overline{n}$ is the largest number of times we can iterate Lemma \ref{lemma:PartialHarnackIneqBranchingFull}.

    First, we deal with \eqref{eqn:HolderTildeunBranching}. To this aim, let $\overline{x} \in \overline{\Omega_{u_n}} \cap B_{1/2}$, $\rho, c$ as in Lemma \ref{lemma:PartialHarnackIneqBranchingFull} and $\sigma \ge -\varepsilon$ so that
    \begin{align*}
        & \sqrt{\Lambda_u} \left( x_N + \sigma \right)^+ \le u_{1/2} \le \sqrt{\Lambda_u} \left( x_N + \sigma + \varepsilon\right)^+ \quad \text{in } B\left( \overline{x} \right) , \\
        & \sqrt{\Lambda_v} \left( x_N + \sigma \right)^+ \le v_{1/2} \le \sqrt{\Lambda_v} \left( x_N + \sigma + \varepsilon\right)^+ \quad \text{in } B\left( \overline{x} \right) ,
    \end{align*}
    which exists thanks to \eqref{eqn:ImprovFlatnessHp} and \eqref{eqn:ImprovFlatRegimeBranching}. Now we can iterate Lemma \ref{lemma:PartialHarnackIneqBranchingFull} $k \in \N \cup \{ 0 \}$ times, until one of the following occurs.
    \begin{itemize}
        \item[1.] $k = \overline{n}$ ,
        
        \item[2.] $\left\vert \sigma_{u_{1/2 \rho^k}} \right\vert > 1/10$ or $\left\vert \sigma_{v_{1/2 \rho^k}} \right\vert > 1/10$ (which are not mutually exclusive).
    \end{itemize}
    Now, case 1 implies that
    \begin{equation}\label{eqn:compactnessBranchingUToProve}
    \sqrt{\Lambda_u} \left( x_N + \sigma_{u, k} \right)^+ \le u \le \sqrt{\Lambda_u} \left( x_N + \sigma_{u, k} + (1-c)^k\varepsilon\right)^+ \quad \text{in } B_{\rho^k/2}\left( \overline{x} \right) , \quad k = 0, ..., \overline{n}
    \end{equation}
    so that
    \[
    \frac{\left\vert \tilde{u}_n\left( x \right) - \tilde{u}_n\left( \overline{x} \right) \right\vert}{\rho^{k+1}} \le \frac{2}{\rho}  \left(\frac{1-c}{\rho}\right)^k \quad \text{for } x \in \overline{\Omega_{u_n}} \cap \left( B_{\rho^k/2}\left( \overline{x} \right) \setminus B_{\rho^{k+1}/2}\left( \overline{x} \right) \right), \quad k = 0, ..., \overline{n}
    \]
    which implies \eqref{eqn:HolderTildeunBranching} with $\beta$ so that $\rho^{\beta} = 1-c$.

    Suppose instead that case 2 occurs. Since $\overline{x} \in \overline{\Omega_{u_n}}$ and by Lemma \ref{lemma:PartialHarnackIneqBranchingFull}
    \[
    \sigma_{u_{1/2 \rho^k}} \ge \sigma_{v_{1/2 \rho^k}} ,
    \]
    necessarily
    \begin{equation}\label{eqn:CompactnessUCase2Characheriz}
    \sigma_{u_{1/2 \rho^k}}  > 1/10 \quad \text{or} \quad \sigma_{v_{1/2 \rho^k}}  < - 1/10
    \end{equation}
    If the first inequality holds, for all steps $k, ..., \overline{n}$ one can apply the standark Harnack inequality to $u_{\rho^{-k}/2}$ to deduce \eqref{eqn:compactnessBranchingUToProve}, so that \eqref{eqn:HolderTildeunBranching} follows again. On the other hand, if the second inequality occurs, for all steps $k, ..., \overline{n}$ one can apply to $u_{\rho^{-k}/2}$ the partial Harnack inequality for the one-phase Bernoulli problem \cite[Lemma 3.3]{DeSilva:FreeBdRegularityOnePhase}, deducing \eqref{eqn:compactnessBranchingUToProve} and consequently \eqref{eqn:HolderTildeunBranching}.

    Now we sketch the proof \eqref{eqn:HolderTildevnBranching}, which is very similar to that of \eqref{eqn:HolderTildeunBranching}. Thus time, let $\overline{x} \in \overline{\Omega_{v_n}} \cap B_{1/2}$, and again $\rho, c$. Then, one can proceed exactly as for \eqref{eqn:HolderTildeunBranching}, introducing cases 1 and 2 as above, the only difference being that conditions \eqref{eqn:CompactnessUCase2Characheriz} are substituted by
    \[
    \sigma_{v_{1/2 \rho^k}}  > 1/10 \quad \text{or} \quad \sigma_{u_{1/2 \rho^k}}  > 1/10.
    \]
    Again, if the first case occurs one can proceed applying for $k, ..., \overline{n}$ the standard Harnack inequality for harmonic functions to $v_{\rho^{-k}/2}$, while if the second case occurs one can proceed applying  to $v_{\rho^{-k}/2}$ Harnack's inequality for the one-phase Bernoulli problem \cite[Lemma 3.3]{DeSilva:FreeBdRegularityOnePhase} for $k, ..., \overline{n}$.
\end{proof}
Now we deal with Lemma \ref{lemma:compactnessCoincidence}, hence with the coincidence regime. The strategy is analogous to that used for Lemma \ref{lemma:compactnessBranching}. Again, the compactness result is based on the following Hölder continuity lemma and an Ascoli-Arzelà compatness argument. We omit the proof of the last part, for which we refer e.g. to \cite{DeSilva:FreeBdRegularityOnePhase} or \cite[Lemma 7.15]{Velichkov:RegularityOnePhaseFreeBd}.
\begin{lemma}\label{lemma:HolderTildeunTildevnCoincidence}
    Under the hypotheses of Lemma \ref{lemma:compactnessCoincidence} there exist constants $c>0$ and $0 < \beta \le 1$ such that for all $\overline{x} \in \overline{\Omega_{u_n}} \cap B_{1/2}$ and $x \in \overline{\Omega_{u_n}} \cap \left( B_{1/2}\left( \overline{x} \right) \setminus B_{\varepsilon_n/\varepsilon_0}\left( \overline{x} \right) \right)$
    \begin{equation}\label{eqn:HolderTildeqnCoincidence}
        \left\vert \tilde{u}_n\left( x \right) - \tilde{u}_n\left( \overline{x} \right) \right\vert \le c \left\vert  x - \overline{x} \right\vert^{\beta} \quad \text{and} \quad \left\vert \tilde{v}_n\left( x \right) - \tilde{v}_n\left( \overline{x} \right) \right\vert \le c \left\vert  x - \overline{x} \right\vert^{\beta}.
    \end{equation}
\end{lemma}
\begin{proof}
    We basically iterate Lemma \ref{lemma:PartialHarnackIneqCoincidence}. Let $n \in \N$ be fixed and choose $\overline{n}$ as the largest integer such that $\varepsilon_n \le \rho^{\overline{n}} \varepsilon_0$, with $\rho$ and $\varepsilon_0$ as in Lemma \ref{lemma:PartialHarnackIneqCoincidence}. In the following we denote $u_u = u$ and $v_n = v$.
    
    Let $\overline{x} \in \overline{\Omega_{q_n}} \cap B_{1/2}$ and $\sigma \ge - \varepsilon$ so that
    \[
    \left( x_N + \sigma \right)^+ \le u_{1/2} \le \left( x_N + \sigma + \varepsilon\right)^+ \quad \text{and} \quad \left( x_N + \sigma \right)^+ \le v_{1/2} \le \left( x_N + \sigma + \varepsilon\right)^+ \quad \text{in } B\left( \overline{x} \right) ,
    \]
    which exists thanks to \eqref{eqn:ImprovFlatnessHp}. Before iterating Lemma \ref{lemma:PartialHarnackIneqCoincidence}, we remark that the main key difference with respect to the proof of Lemma \ref{lemma:compactnessBranching} is that, for all iterations,
    \[
    \sigma_{u_{1/2 \rho^k}} = \sigma_{v_{1/2 \rho^k}} \coloneqq \sigma_k.
    \]
    Now we iterate Lemma \ref{lemma:PartialHarnackIneqCoincidence} $k$ times, until one of the following occurs.
    \begin{itemize}
        \item[1.] $k = \overline{n}$ ,
        
        \item[2.] $\vert \sigma_k \vert > 1/10$.
    \end{itemize}
    Case 1 gives, for all $k = 0, ..., \overline{n}$
    \begin{align}\label{eqn:compactnessCoincidenceToProve}
    \begin{split}
        & \left( x_N + \sigma_{u, k} \right)^+ \le u \le \left( x_N + \sigma_{u, k} + (1-c)^k\varepsilon\right)^+ \quad \text{in } B_{\rho^k/2}\left( \overline{x} \right) , \\
        & \left( x_N + \sigma_{v, k} \right)^+ \le v \le \left( x_N + \sigma_{v, k} + (1-c)^k\varepsilon\right)^+ \quad \text{in } B_{\rho^k/2}\left( \overline{x} \right) ,
    \end{split}
    \end{align}
    which in turn imply \eqref{eqn:HolderTildeqnCoincidence} with $\beta$ such that $\rho^{\beta} = 1-c$.

    If case 2 occurs, since $\overline{x} \in \overline{\Omega_{u_n}}$, necessarily  $\sigma_k > 1/10$. Hence, for all steps $k, ..., \overline{n}$ one can apply to $u_{\rho^{-k}/2}$ and $v_{\rho^{-k}/2}$ the standard Harnack inequality for harmonic functions, to deduce \eqref{eqn:compactnessCoincidenceToProve} and consequently \eqref{eqn:HolderTildeqnCoincidence}.
\end{proof}

\subsubsection{Limit problem}\label{subsec:LimitPb}
In this section we derive the linearized problems solved by the limit functions $\tilde{u}_{\infty}$, $\tilde{v}_{\infty}$ and $\tilde{q}_{\infty}$ defined in Lemmas \ref{lemma:compactnessBranching} and \ref{lemma:HolderTildeunTildevnCoincidence} respectively. Also in this case, we distinguish the branching and the coincidence regime.

Before giving the results, some definitions might be useful.
\begin{definition}\label{def:ViscousOneSidedTwoMembranePb}
Let the functions $h, w \in C^{0}\left( \overline{B_1^+} \right)$. We say that they are viscosity solutions of the 
 one-sided two membranes problem with coefficients $\Lambda_h$ and $\Lambda_w$
\begin{equation}\label{eqn:ViscousOneSidedTwoMembranePb}
    \begin{cases}
        \Delta h = \Delta w = 0 & \text{in } B_1^+ , \\
         h \ge w & \text{on } \{ x_N = 0 \} \cap \overline{B_1^+} , \\
         \partial_N h \le 0 \text{ and } \partial_N w \ge 0 & \text{on } \{ x_N = 0 \} \cap B_1 , \\
         \partial_N h = \partial_N w = 0 & \text{on } \{ h > w \} \cap \{ x_N = 0 \} \cap B_1 , \\
         \Lambda_h \partial_N h  + \Lambda_w \partial_N w = 0 & \text{on } \{ h = w \} \cap \{ x_N = 0 \} \cap B_1 ,
    \end{cases}
\end{equation}
if the following properties hold:
\begin{itemize}
    \item[(i)] The functions $h$ and $w$ are harmonic in the viscosity sense in $B_1^+$,

    \item[(ii)] Whenever a second order polynomial $p(x)$ touches $h \, (w)$ from below (above) at $\overline{x} \in \{ x_N = 0 \} \cap B_1$, then $\partial_N \varphi \left( \overline{x} \right) \le 0 \, (\ge 0)$,

    \item[(iii)] Whenever a second order polynomial $p(x)$ touches $h \, (w)$ from above (below) at $\overline{x} \in \{ h > w \} \cap \{ x_N = 0 \} \cap B_1$, then $\partial_N \varphi \left( \overline{x} \right) \ge 0 \, (\le 0)$,

    \item[(iv)] Let $A \le 0$, $B \ge 0$, $p_u(x) \coloneqq A x_N + p(x)$ and $p_v(x) \coloneqq B x_N + p(x)$ touch respectively $u$ and $v$ from above (below) at $\overline{x} \in \{ h = w \} \cap \{ x_N = 0 \} \cap B_1$, for some second order polynomial $p(x)$ with $\partial_N p\left( \overline{x} \right) = 0$. Then $\Lambda_h A+ \Lambda_w B \le 0 \, (\Lambda_h A+\Lambda_w B \ge 0)$.
\end{itemize}
\end{definition}
\begin{definition}\label{def:ViscousOneSidedTransmissionPb}
Let the functions $h, w \in C^{0}\left( \overline{B_1^+} \right)$. We say that they are viscosity solutions of the following one-sided transmission problem (with coefficients $\Lambda_h$ and $\Lambda_w$)
\begin{equation}\label{eqn:ViscousOneSidedTransmissionPb}
    \begin{cases}
        \Delta h = \Delta w = 0 & \text{in } B_1^+ , \\
         h = w & \text{on } \{ x_N = 0 \} \cap B_1 , \\
         \partial_N h \le 0 \text{ and } \partial_N w \ge 0 & \text{on } \{ x_N = 0 \} \cap B_1 , \\
         \Lambda_h \partial_N h  + \Lambda_w \partial_N w = 0 & \text{on } \{ x_N = 0 \} \cap B_1 ,
    \end{cases}
\end{equation}
if the following properties hold:
\begin{itemize}
    \item[(i)] The functions $h$ and $w$ are harmonic in the viscosity sense in $B_1^+$,

    \item[(ii)] If a second order polynomial $p(x)$ touches $h \, (w)$ from below (above) at $\overline{x} \in \{ x_N = 0 \} \cap B_1$, then $\partial_N \varphi \left( \overline{x} \right) \le 0 \, (\ge 0)$,

    \item[(iii)] Let $A \le 0$, $B \ge 0$, $p_u(x) \coloneqq A x_N + p(x)$ and $p_v(x) \coloneqq B x_N + p(x)$ touch respectively $u$ and $v$ from above (below) at $\overline{x} \in \{ x_N = 0 \} \cap B_1$, for some second order polynomial $p(x)$ with $\partial_N p\left( \overline{x} \right) = 0$. Then $\Lambda_h A+ \Lambda_w B \le 0 \, (\Lambda_h A+\Lambda_w B \ge 0)$.
\end{itemize}
\end{definition}
\begin{remark}\label{rmk:ViscousLimitPbsRegularity}
We recall that, if the functions $h, w \in C^{0}\left( \overline{B_1^+} \right)$ are viscosity solutions to problem \eqref{eqn:ViscousOneSidedTwoMembranePb} in the sense of definition \ref{def:ViscousOneSidedTwoMembranePb}, then $h, w \in C^{1, 1/2}\left( B_1^+ \right)$ up to $\{ x_N = 0 \}$. Moreover, there exists a constant $C>0$ independent of $h$ and $w$ such that
\[
\| h \|_{C^{1/2}\left( \overline{B}_{1/2}^+ \right)} \le C \| h \|_{L^{\infty}\left( \overline{B}_{1}^+ \right)} , \quad \text{and} \quad \| w \|_{C^{1/2}\left( \overline{B}_{1/2}^+ \right)} \le C \| w \|_{L^{\infty}\left( \overline{B}_{1}^+ \right)} .
\]
This can be seen splitting $h$ and $w$ respectively into a sum and a difference of a harmonic function in $B_1$ and a solution of the thin obstacle problem (see e.g. \cite[Appendix B]{DePhilippisSpolaorVelichkov2021:TwoPhaseBernoulli}).

If, on the other hand, the functions $h, w \in C^{0}\left( \overline{B_1^+} \right)$ are viscosity solutions to problem \eqref{eqn:ViscousOneSidedTransmissionPb} in the sense of definition \ref{def:ViscousOneSidedTransmissionPb}, then $h, w \in C^{\infty}\left( B_1^+ \right)$ up to $\{ x_N = 0 \}$. Similarly as before, there exists a constant $C>0$ independent of $h$ and $w$ such that
\[
\| h \|_{C^{2}\left( \overline{B}_{1/2}^+ \right)} \le C \| h \|_{L^{\infty}\left( \overline{B}_{1}^+ \right)} , \quad \text{and} \quad \| w \|_{C^{2}\left( \overline{B}_{1/2}^+ \right)} \le C \| w \|_{L^{\infty}\left( \overline{B}_{1}^+ \right)} .
\]
This result can be found in \cite[Section 3]{DeSilvaFerrariSalsa:2PhaseFreeBdDivergenceForm}
\end{remark}
We begin studying the branching regime, for which we have the following
\begin{lemma}\label{lemma:linearizationBranching}
    Let the functions $\tilde{u}_{\infty}$ and $\tilde{v}_{\infty}$ be as in Lemma \ref{lemma:compactnessBranching}. Then, in the sense of Definition \ref{def:ViscousOneSidedTwoMembranePb}, they are viscosity solutions of problem \eqref{eqn:ViscousOneSidedTwoMembranePb} (with $h = \tilde{u}_{\infty}$, $w = \tilde{v}_{\infty}$, $\Lambda_h = \Lambda_u$ and $\Lambda_w = \Lambda_v$).
\end{lemma}
\begin{proof}
    The proof of Definition \eqref{def:ViscousOneSidedTwoMembranePb} (i), (ii), (iii) is standard and we refer to \cite{DeSilva:FreeBdRegularityOnePhase, Velichkov:RegularityOnePhaseFreeBd, DePhilippisSpolaorVelichkov2021:TwoPhaseBernoulli, DeSilvaFerrariSalsa:2PhaseFreeBdDivergenceForm}. Now we prove Definition \eqref{def:ViscousOneSidedTwoMembranePb} (iv). 
    
    Let $A, B, p_u(x), p_v(x), p(x)$ and $\overline{x}$ as in Definition \eqref{def:ViscousOneSidedTwoMembranePb} (iv), with $p_u$ and $p_v$ touching respectively $\tilde{u}_{\infty}$ and $\tilde{v}_{\infty}$ from below.
    
    Without loss of generality we can assume that they touch strictly from below and that $\Delta p > 0$, and consequently $\Delta p_u > 0$ and $\Delta p_v > 0$. We need to show that $\Lambda_h A+ \Lambda_w B \le 0$, and we do this by contradiction.

    Suppose that $\Lambda_u A + \Lambda_v B > 0$. Let us consider a new competitor $\tilde{p}_v(x)$ touching $\tilde{v}_{\infty}$ strictly from below at the same point $\overline{x}$, defined by
    \[
    \tilde{p}_v(x) \coloneqq \left( \Lambda_u A + \Lambda_v B \right) x_N + p(x).
    \]
    and let us denote
    \[
    \phi_n(x) \coloneqq \sqrt{\Lambda_v} x_N + \varepsilon_n \sqrt{\Lambda_v} \tilde{p}_v(x)
    \]
    By Lemma \ref{lemma:compactnessBranching} there exist a small constant $\delta>0$ and constants $\eta_n \to 0$ as $n \to +\infty$, such that for $n \in \N$ sufficiently large
    \begin{itemize}
        \item[(i)] $\phi_n(x - t e_N) < u_n$ on $\partial B_{\delta}\left( \overline{x} \right)$ for all $t \in [-\eta_n, \eta_n ]$,

        \item[(ii)] there exists $t_n \in (-\eta_n, \eta_n )$ such that $\phi_n(x - t e_N)$ touches $u_n$ from below at $x_n \in B_{10 \sqrt{\Lambda_v} \eta_n}\left( \overline{x} \right)$ and $\phi_n(x - t e_N) > u_n$ in $B_{\delta}\left( \overline{x} \right)$ for all $t \in (t_n, \eta_n]$,

        \item[(iii)] $x_n \to \overline{x}$ as $n \to +\infty$.
    \end{itemize}
    Now, $x_n$ can be either one-phase points for $v_n$, or two-phase points. Let us reach a contradiction proving that, for $n$ sufficiently large, none of the cases above is possible.
    
    First, we prove that $x_n$ cannot be one-phase points for $v_n$. Indeed, by Definition \ref{def:viscoudSolOneSidedTwoPhasePb} (i) we would have
    \begin{align*}
    \begin{split}
        & \Lambda_u \ge \left\vert \nabla \phi_n(x - t_n e_N)(x_n) \right\vert^2 = \Lambda_u + \varepsilon_n \Lambda_u \partial_N \tilde{p}_v(x_n) + O(\varepsilon_n^2) = \\
        & = \Lambda_u + \varepsilon_n \Lambda_u (\Lambda_u A + \Lambda_v B + \partial_N p(x_n)) + O(\varepsilon_n^2)
    \end{split}
    \end{align*}
    which gives a contradiction for large $n$, since $\partial_N p(x_n) \to 0$ as $n \to +\infty$ while $\Lambda_u A + \Lambda_v B > 0$ by hypothesis. 

    Hence, $x_n$ need to be two phase points, but now we show that this is not possible. To this aim, notice that the competitor $\tilde{p}_v(x)$ touches from below at $\overline{x}$ also the (linearized) average $\tilde{q}_{\infty} \coloneqq \sqrt{\Lambda_u} \tilde{u}_{\infty} + \sqrt{\Lambda_v} \tilde{v}_{\infty}$.

    Now consider the competitors for the average $q_n(x) = \sqrt{\Lambda_u} u_n + \sqrt{\Lambda_v} v_n$ defined by
    \[
    \varphi_n(x) \coloneqq x_N + \varepsilon_n \tilde{p}_v(x).
    \]
    Similarly as for the partial Harnack inequality (Lemma \ref{lemma:PartialHarnackIneqBranching}), the key point is that the competitors $\phi_n$ and $\varphi_n$ for $v_n$ and $q_n$ respectively, share the same free boundary.

    By Lemma \ref{lemma:compactnessBranching} and the previous discussion we have that 
    \begin{itemize}
        \item[(i)] $\varphi_n(x - t e_N) < q_n$ on $\partial B_{\delta}\left( \overline{x} \right) \cap \{ v_n > 0 \}$ for all $t \in [-\eta_n, \eta_n ]$,

        \item[(ii)] $\varphi_n(x - \eta_n e_N) < q_n$ in $\overline{B_{\delta}\left( \overline{x} \right)}$.
    \end{itemize}
    Notice that in point $(ii)$ the restriction $x \in \{u > 0 \}$ has been dropped, since for $t \in [t_n, \eta_n]$ the free boundary of $\varphi_n(x - \eta_n e_N)$ belongs to $\overline{\{ v_n > 0 \}}$ and consequently $\{ \varphi_n(x - \eta_n e_N) > 0 \} \cap B_{\delta}\left( \overline{x} \right) \subseteq \{ v_n > 0 \}$.
    
    For the same reason we have that $\varphi_n(x-t_n e_N)$ touches $q_n$ from below at $x_n$ and $\varphi_n(x-t_n e_N) < q_n$ in $B_{\delta}\left( \overline{x} \right)$ for all $t \in (t_n, \eta_n]$ (notice that $\varphi_n(x-t_n e_N)$ cannot touch $q_n$ at a point in $\{ q_n > 0 \} \cap \{ v_n>0 \}$, since $q_n$ is harmonic on such set). However, this leads to a contradiction. Indeed, by \ref{def:viscoudSolOneSidedTwoPhasePb} (i) we would have
    \begin{align*}
    \begin{split}
        & 1 \ge \left\vert \nabla \phi_n(x - t_n e_N)(x_n) \right\vert^2 = 1 + \varepsilon_n \partial_N \tilde{p}_v(x_n) + O(\varepsilon_n^2) = \\
        & = 1+ \varepsilon_n (\Lambda_u A + \Lambda_v B + \partial_N p(x_n)) + O(\varepsilon_n^2)
    \end{split}
    \end{align*}
    which again gives a contradiction for large $n$, since $\partial_N p(x_n) \to 0$ and $\Lambda_u A + \Lambda_v B > 0$.

    Now we deal with the case when $p_u$ and $p_v$ touch $\tilde{u}_{\infty}$ and $\tilde{v}_{\infty}$ from above at $\overline{x}$. We only sketch the proof, since the strategy is analogous to the case just treated when they touch from below.

     Without loss of generality we can assume that the polynomials touch strictly from above and that $\Delta p$, $\Delta p_u$, $\Delta p_v > 0$.
     
     Again, we proceed by contradiction. Assume that $\Lambda_u A + \Lambda_v B < 0$ and introduce the competitors for $\tilde{u}_{\infty}$ defined by
    \[
    \tilde{p}_{u, n}(x) \coloneqq \left( \Lambda_u A + \Lambda_v B \right) x_N + p(x)
    \]
    which touch $\tilde{u}_{\infty}$ from above at the same point $\overline{x}$. Notice also that $\Delta \tilde{p}_{u, n} > 0$ for $n$ sufficiently large. Introduce also the following competitors for $u_n$
    \[
    \phi_n(x) \coloneqq \Lambda_u \left( x_N + \tilde{p}_{u, n}(x) + C \varepsilon_n^2 \right)_+.
    \]
    for some suitable constant $C>0$ (independent of $n$) to be precised in the following.
    
    Similarly as before, by Lemma \ref{lemma:compactnessBranching} we have that $\phi_n(x-t_n e_N)$ touches $u_n$ from above at some $x_n \in B_{\delta}\left( \overline{x} \right)$ and also $\{ u_n > 0 \} \Subset \{ \phi_n(x-t e_N) > 0 \}$.

    Again, $x_n$ are either one-phase point of $u_n$ or two phase points. However, since $u_n$ is a viscosity solution of problem \eqref{eqn:MainViscousPb}, similarly as in the previous case the points $x_n$ are necessarily two-phase points since $\Lambda_u A + \Lambda_v B < 0$ by hypothesis.

    Now, since
    \[
    \sqrt{a x^2 + b y^2} - (ax+by) = \frac{ab(x-y)^2}{\sqrt{a x^2 + b y^2} + ax+by} \quad \text{where } a, b, x, y >0 \text{ and } a+b=1 ,
    \]
    using the above identity with $a = \Lambda_u$, $b = \Lambda_v$, $x = \left( x_N + \varepsilon_N \left( A x_N + p \right) \right)_+$ and $y = \left( x_N + \varepsilon_N \left( B x_N + p \right) \right)_+$
    together with the strict monotonicity of the denominator and the fact that the free boundaries of $x$ and $y$ in direction $e_N$ are distant less then $c \varepsilon_n^2$ (for some $c>0$ independent of $n$), we have that in $B_{\delta}\left( \overline{x} \right)$
    \begin{align*}
        \begin{split}
            & \sqrt{\Lambda_u \left( x_N + \varepsilon_N \left( A x_N + p \right) \right)_+^2 + \Lambda_v \left( x_N + \varepsilon_N \left( B x_N + p \right) \right)_+^2} \le \\
            & \le \left( x_N + \varepsilon_n \left( \left( \Lambda_u A + \Lambda_v B \right) x_N + p(x) \right) + C \varepsilon_n^2 \right)_+
        \end{split}
    \end{align*}
    for some constant $C>0$ independent of $n$, which is the  choice that we make.
    
    Hence, we see that
    \[
    \varphi_n(x) \coloneqq \left( x_N + \tilde{p}_{u, n}(x) + C \varepsilon_n^2 \right)_+
    \]
    is a competitor for the modulus $m_n \coloneqq \sqrt{u_n^2+v_n^2}$ such that
    \begin{itemize}
        \item[(i)] $\varphi_n(x - t e_N) > m_n$ on $\partial B_{\delta}\left( \overline{x} \right) \cap \{ u_n > 0 \}$ for all $t \in [-\eta_n, \eta_n ]$,

        \item[(ii)] $\varphi_n(x - \eta_n e_N) > m_n$ in $\overline{B_{\delta}\left( \overline{x} \right)} \cap \{ u_n > 0 \}$.
    \end{itemize}
    Property (ii) is trivial, while one can check property (i) first on $B_{\delta}\left( \overline{x} \right) \cap \left( \{ u_n > 0 \} \setminus \{ v_n>0 \} \right)$, where $m_n = u_n$ and it follows from the analogous property for $\phi_n(x)$, and then on $B_{\delta}\left( \overline{x} \right) \cap  \{ v_n>0 \} $ where it follows from Lemma \ref{lemma:compactnessBranching}.

    Again, the fundamental point is that $\phi_n$ and $\varphi_n$ share the same free boundary, hence $\phi_n$ cannot touch $u_n$ at a two-phase point $x_n$, otherwise $\varphi_n$ would touch $m_n$ at the same point which, using Definition \ref{def:viscoudSolOneSidedTwoPhasePb} (ii) would lead to a contradiction since $\Lambda_u A + \Lambda_v B < 0$ by hypothesis.
\end{proof}
Now we turn to the coincidence regime. Before stating the result we need to introduce some notation. We denote $\Gamma_{u, \infty}$ and $\Gamma_{v, \infty}$ two real numbers such that $\Gamma_{u, n} \to \Gamma_{u, \infty}$ and $\Gamma_{v, n} \to \Gamma_{v, \infty}$ as $n \to +\infty$, which exist up to a subsequence by Remark \ref{rmk:BddGammauvCoincidence}.

We are ready to state the following
\begin{lemma}\label{lemma:linearizationCoincidence}
    Let the functions $\tilde{u}_{\infty}$ and $\tilde{v}_{\infty}$ be as in Lemma \ref{lemma:compactnessCoincidence}. Then, they are viscosity solutions of problem \eqref{eqn:ViscousOneSidedTransmissionPb} (with $h = \tilde{u}_{\infty}$, $w = \tilde{v}_{\infty}$, $\Lambda_h = \Gamma_{u, \infty}^2$ and $\Lambda_w = \Gamma_{v, \infty}^2$), in the sense of Definition \ref{def:ViscousOneSidedTransmissionPb}.
\end{lemma}
\begin{proof}
    First of all, we notice that $\tilde{u}_{\infty} = \tilde{v}_{\infty}$ on $\{ x_N = 0 \} \cap B_1$ by Lemma \ref{lemma:compactnessCoincidence} (ii).
    
    Now we only sketch the proof which follows the same strategy of that of Lemma \ref{lemma:compactnessBranching}. Again we only deal with Definition \ref{def:ViscousOneSidedTransmissionPb} (iii) and we refer to \cite{DeSilva:FreeBdRegularityOnePhase, Velichkov:RegularityOnePhaseFreeBd, DePhilippisSpolaorVelichkov2021:TwoPhaseBernoulli, DeSilvaFerrariSalsa:2PhaseFreeBdDivergenceForm} for points (i) and (ii).

    Let $\overline{x}, A, B, p_u(x), p_v(x), p(x)$ (with $\Delta p > 0$) as in Definition \eqref{def:ViscousOneSidedTransmissionPb} (iii) and suppose that $p_u$ and $p_v$ touch respectively $\tilde{u}_{\infty}$ and $\tilde{v}_{\infty}$ strictly from below. Assume by contradiction that $\Gamma_{u, \infty} A + \Gamma_{v, \infty} B > 0$ and introduce the following  competitor for $v_n$
    \begin{align*}
    & \phi_n(x) \coloneqq \Gamma_{v, n} \left( x_N + \varepsilon_n \left( \Gamma_{u, n}^2 A + \Gamma_{v, n}^2 B \right) x_N + \varepsilon_n \tilde{p}_v(x) \right) = \\ = & \Gamma_{v, n} \left( x_N + \varepsilon_n (1+o(1)) \left( \Gamma_{u, \infty}^2 A + \Gamma_{v, \infty}^2 B \right) x_N + \varepsilon_n \tilde{p}_v(x)\right),
    \end{align*}
    where $o(1) \to 0$ as $n \to +\infty$. Similarly as in the proof of Lemma \ref{lemma:linearizationBranching}, $\phi_n(x+te_N)$ is a family of barriers from below for $u_n$ in $B_{\delta}\left( \overline{x} \right)$, for $t$ in a suitable interval. Let $\overline{t}$ be the largest $t$ so that $\phi_n(x+te_N)<u_n$ strictly. Then, $\phi_n(x+\overline{t}e_N)$ can only touch $u_n$ at a one or two-phase free boundary point $x_n \to \overline{x}$ as $n \to +\infty$. The condition $\Gamma_{u, \infty} A + \Gamma_{v, \infty} B > 0$ and the fact that $\Gamma_{u, n} \ge \sqrt{\Lambda_v}$ imply that $x_n$ can only be a two-phase free boundary point. However, introducing the competitor
    \[
    \varphi_n(x) \coloneqq x_N + \varepsilon_n (1+o(1)) \left( \Gamma_{u, \infty}^2 A + \Gamma_{v, \infty}^2 B \right) x_N + \varepsilon_n \tilde{p}_v(x)
    \]
    for the average $q_n(x) \coloneqq \Gamma_{u, n} u_n + \Gamma_{v, n} v_n$, so that $\phi_n$ and $\varphi_n$ have the same free boundary, similarly as in Lemma \ref{lemma:linearizationBranching} one can show that $x_n$ cannot be a two-phase point, thus reaching a contradiction.

    Now suppose that $p_u(x), p_v(x)$ touch respectively $\tilde{u}_{\infty}$ and $\tilde{v}_{\infty}$ strictly from above at $\overline{x}$, and that $\Delta p < 0$. If by contradiction we assume that $\Gamma_{u, \infty} A + \Gamma_{v, \infty} B < 0$, then similarly as in Lemma \ref{lemma:linearizationBranching} we can introduce the competitors (for some suitable constant $C>0$ independent of $n$)
    \begin{align*}
        & \phi_n(x) \coloneqq \Gamma_{u, n} \left( x_N + \varepsilon_n (1+o(1)) \left( \Gamma_{u, \infty}^2 A + \Gamma_{v, \infty}^2 B \right) x_N + \varepsilon_n p(x) + C \varepsilon_n^2 \right), \\
        & \varphi_n(x) \coloneqq x_N + \varepsilon_n (1+o(1)) \left( \Gamma_{u, \infty}^2 A + \Gamma_{v, \infty}^2 B \right) x_N + \varepsilon_n p(x) + C \varepsilon_n^2
    \end{align*}
    for $u_n$ and the modulus $m_n \coloneqq \sqrt{u_n^2 + v_n^2}$ respectively. For suitable values of $t$, $\phi_n(x+te_n)$ and $\varphi_n(x+te_n)$ are barriers from above for $u_n$ and $m_n$ respectively and they share the same free boundary. Moreover, for some $\overline{t}$ they touch the free boundary of $u_n$ and $m_n$ at the same point, which consequently cannot be a one-phase nor a two-phase point. Again a contradiction.
\end{proof}

\subsubsection{Improvement}\label{subsubsec:improvement}
In this section we complete the proof of the improvement of flatness, Theorem \ref{thm:ImprovFlat}. We follow \cite{DeSilva:FreeBdRegularityOnePhase} and proceed by contradiction.

\begin{proof}[Proof of Theorem \ref{thm:ImprovFlat}.]
Without loss of generality we can assume $\nu = e_N$. Consider sequences of non-negative real numbers $\Gamma_{u, n}$, $\Gamma_{v, n}$ and functions $u_n$ and $v_n$ as in Theorem \ref{thm:ImprovFlat}, but so that the thesis is false. Thanks to Remark \ref{rmk:BddGammauvCoincidence}, without loss of generality we can assume that $\Gamma_{u, n}$ and $\Gamma_{v, n}$ are bounded from above and below by positive constants, depending only on $\sqrt{\Lambda_u}$ and $\sqrt{\Lambda_v}$.    

We treat the branching and coincidence regimes at the same time, since the proof in both cases is basically identical.

Up to a subsequence one between conditions \eqref{eqn:ImprovFlatRegimeBranching} and \eqref{eqn:ImprovFlatRegimeCoincidence} is satisfied. Hence, by Lemma \ref{lemma:compactnessBranching} or \ref{lemma:compactnessCoincidence} there exist two limit functions $\tilde{u}_{\infty}$ and $\tilde{v}_{\infty}$ which, by Lemma \ref{lemma:linearizationBranching} or \ref{lemma:linearizationCoincidence} are viscosity solutions of problem \eqref{eqn:ViscousOneSidedTwoMembranePb} or \eqref{eqn:ViscousOneSidedTransmissionPb} respectively.

By Remark \ref{rmk:ViscousLimitPbsRegularity} we have (at least) that  $\tilde{u}_{\infty}$, $\tilde{v}_{\infty} \in C^{1, 1/2}\left( \overline{B_{1/2}^+} \right)$ with the $C^{1, 1/2}$ norm controlled by the $L^{\infty}$ norm. Using this fact and that $ 0 \in \partial\{ u_n>0 \} \cap \partial\{ v_n>0 \} $, there exist $\rho, C>0$ universal such that
\[
\left\vert \tilde{u}_{\infty}(x) - \nabla \tilde{u}_{\infty}(0) \cdot x \right\vert \le C \rho^{3/2} \quad \text{and} \quad \left\vert \tilde{v}_{\infty}(x) - \nabla \tilde{v}_{\infty}(0) \cdot x \right\vert \le C \rho^{3/2} \quad \text{for all } x \in B_{\rho}.
\]
For notational simplicity, let us introduce $\nu_u \coloneqq \nabla \tilde{u}_{\infty}(0)$, $\nu_u \coloneqq \nabla \tilde{u}_{\infty}(0)$, $\overline{\nu_u} \coloneqq (\nu_{u, 1}, ..., \nu_{u, N-1})$ and $\overline{\nu_v} \coloneqq (\nu_{v, 1}, ..., \nu_{v, N-1})$. By Remark \ref{rmk:ViscousLimitPbsRegularity} we have
\begin{equation}\label{eqn:ImprovFlatPfNormals}
    \Gamma_{u, \infty} \nu_u \cdot e_n + \Gamma_{v, \infty} \nu_v \cdot e_n = 0 \quad \text{and} \quad \overline{\nu_u} = \overline{\nu_v}.
\end{equation}
By Lemma \ref{lemma:compactnessBranching} (ii) or Lemma \ref{lemma:compactnessCoincidence} (ii) we infer that (possibly modifying C, still universal), for $n \in \N$ sufficiently large
\begin{align}\label{eqn:ImrpovFlatPfIntermedBound}
\begin{split}
    & \Gamma_{u, n} \left( x_N + \varepsilon_n \nu_u \cdot x - C \varepsilon_n \rho^{3/2} \right) \le u_n(x) \le \Gamma_{u, n} \left( x_N + \varepsilon_n \nu_u \cdot x + C \varepsilon_n \rho^{3/2} \right) \quad \text{in } B_{\rho} \cap \overline{\Omega_{u_n}}, \\
    & \Gamma_{v, n} \left( x_N + \varepsilon_n \nu_v \cdot x - C \varepsilon_n \rho^{3/2} \right) \le v_n(x) \le \Gamma_{v, n} \left( x_N + \varepsilon_n \nu_v \cdot x + C \varepsilon_n \rho^{3/2} \right)  \quad \text{in } B_{\rho} \cap \overline{\Omega_{v_n}} .
\end{split}
\end{align}
In the branching regime one can actually take $\Gamma_{u, n} = \sqrt{\Lambda_u}$ and $\Gamma_{v, n} = \sqrt{\Lambda_v}$. Now introduce
\[
\Gamma_{u, n}' \coloneqq \Gamma_{u, n} (1+ \varepsilon_n \nu_{u, N}) + o(\varepsilon_n) c_{u, n} \quad \text{and} \quad \Gamma_{v, n}' \coloneqq \Gamma_{v, n} (1+ \varepsilon_n \nu_{v, N}) + o(\varepsilon_n) c_{v, n} ,
\]
where $c_{u, n}$ and $c_{v, n}$ are constants (uniformly bounded in $n$ by a universal constant $C>0$) chosen such that
\[
\Gamma_{u, n}^{'2} + \Gamma_{v, n}^{'2} = 1 .
\]
This choice is possible since by \eqref{eqn:ImprovFlatnessHp} and \eqref{eqn:ImprovFlatPfNormals}
\begin{align*}
& \Gamma_{u, n}^2 (1+ \varepsilon_n \nu_{u, N})^2 + \Gamma_{v, n}^2 (1+ \varepsilon_n \nu_{v, N})^2 = 1 + 2 \varepsilon_n \left(\Gamma_{u, n}^2 \nu_{u, N} + \Gamma_{v, n}^2 \nu_{v, N} \right) + c_n \varepsilon_n^2  = \\
& = 1 + 2 \varepsilon_n \left(\Gamma_{u, \infty}^2 \nu_{u, N} + \Gamma_{v, \infty}^2 \nu_{v, N} \right) + c_n  \varepsilon_n o(1) = 1 + c_n \varepsilon_n o(1) ,
\end{align*}
for some constants $c_n$ bounded uniformly in $n$ by $C>0$ universal. Now, we have
\begin{align*}
    & x_N + \varepsilon_n \nu_u \cdot x = (\varepsilon_n \overline{\nu_u}, 1 + \varepsilon_n \nu_{u, N}) \cdot x \quad \text{and} \quad x_N + \varepsilon_n \nu_v \cdot x = (\varepsilon_n \overline{\nu_v}, 1 + \varepsilon_n \nu_{v, N}) \cdot x , \\
    & \left\vert \Gamma_{u, n} (\varepsilon_n \overline{\nu_u}, 1 + \varepsilon_n \nu_{u, N}) \cdot x - \Gamma_{u, n}' \frac{(\varepsilon_n \overline{\nu_u}, 1)}{\vert (\varepsilon_n \overline{\nu_u}, 1) \vert} \cdot x \right\vert \le C o(\varepsilon_n) \rho \quad \text{in } B_{\rho}, \\
    & \left\vert \Gamma_{v, n} (\varepsilon_n \overline{\nu_v}, 1 + \varepsilon_n \nu_{v, N}) \cdot x - \Gamma_{u, n}' \frac{(\varepsilon_n \overline{\nu_v}, 1)}{\vert (\varepsilon_n \overline{\nu_v}, 1) \vert} \cdot x \right\vert \le C o(\varepsilon_n) \rho \quad \text{in } B_{\rho} ,
\end{align*}
for some universal constant $C > 0$. Denoting
    \[
    \nu' \coloneqq \frac{(\varepsilon_n \overline{\nu_u}, 1 )}{\vert (\varepsilon_n \overline{\nu_u}, 1) \vert} =  \frac{(\varepsilon_n \overline{\nu_v}, 1)}{\vert (\varepsilon_n \overline{\nu_v}, 1) \vert},
    \]
    where equality holds thanks to \eqref{eqn:ImprovFlatPfNormals}, for any fixed $0< \alpha <1/2$, choosing $\rho$ sufficiently small depending only on $C$ and $n$ sufficiently large, conditions \eqref{eqn:ImrpovFlatPfIntermedBound} can be rewritten as
    \begin{align*}
    \begin{split}
        & \Gamma_{u, n}' \left( x_N + \varepsilon_n \nu_u' \cdot x - \varepsilon_n \rho^{1+\alpha} \right) \le u_n(x) \le \Gamma_{u, n}' \left( x_N + \varepsilon_n \nu_u' \cdot x + \varepsilon_n \rho^{1+\alpha} \right) \quad \text{in } B_{\rho} \cap \overline{\Omega_{u_n}} , \\
        & \Gamma_{v, n}' \left( x_N + \varepsilon_n \nu_v' \cdot x - \varepsilon_n \rho^{1+\alpha} \right) \le v_n(x) \le \Gamma_{v, n}' \left( x_N + \varepsilon_n \nu_v' \cdot x + \varepsilon_n \rho^{1+\alpha} \right) \quad \text{in } B_{\rho} \cap \overline{\Omega_{v_n}} ,
    \end{split}
    \end{align*}
    so that conditions \eqref{eqn:ImprovFlatThmSum1}, \eqref{eqn:ImprofFlatThmErrors}, \eqref{eqn:ImprovFlatThmU}, \eqref{eqn:ImprovFlatThmV} are satisfied and we have reached a contradiction.
\end{proof}

\subsection{Proof of the almost-optimal regularity}\label{subsection:AlmostOptimalRegularityProof}
In this section we complete the proof of Theorem \ref{thm:C1aRegularity}.

\begin{proof}[Proof of Theorem \ref{thm:C1aRegularity}.]
    The statement on the regular part follows from Theorem \ref{thm:ImprovFlat} after a standard argument, see e.g. \cite[Section 8.2]{Velichkov:RegularityOnePhaseFreeBd}.

    Concerning the singular set, the existence of a critical dimension follows by the usual dimension reduction argument (see e.g. \cite[Section 10]{Velichkov:RegularityOnePhaseFreeBd}). Moreover, thanks to Lemma \ref{lemma:TwoPhaseBlowUpProportional} we have that all the singular cones are actually singular cones for the one-phase Bernoulli problem. Hence, the bounds from below for the critical dimension come from \cite{CaffarelliJerisonKenig04:NoSingularCones3D} ($N^{\ast} \ge 4$) and \cite{JerisonSavin15:NoSingularCones4D} ($N^{\ast} \ge 5$), while the bound from above comes from \cite{DeSilvaJerison09:SingularConesIn7D} ($N^{\ast} \le 7$).
\end{proof}

\section{Sharp regularity}\label{section:SharpRegularity}
In this section we prove Theorem \ref{thm:SharpRegularity}. Following \cite{ChangLaraSavin:BoundaryRegularityOnePhase}, we do this via Almgren's frequency formula. Since in our case the obstacle is not prescribed a priori, in order to obtain the optimal regularity we need to combine the properties of both inner and outer solutions at the two-phase free boundary. Even though the general strategy is that of \cite{ChangLaraSavin:BoundaryRegularityOnePhase}, several modifications are needed in order to take into account of these differences. For this reason, we give a self-contained proof.

Let $u, v$ be solutions of problem \eqref{eqn:MainViscousPb}. Given a regular point $x_0$ of $\partial \Omega_u \cap \partial \Omega_v$, define the height function as
\[
H(r) \coloneqq \frac{1}{r^{N-1}} \int_{\{u>0\} \cap \partial B_r(x_0)} \left(u - \nabla u (x_0) \cdot x\right)^2 + \frac{1}{r^{N-1}} \int_{\{v>0\} \cap \partial B_r(x_0)} \left(v - \nabla v(x_0) \cdot x\right)^2
\]
and introduce the frequency formula, defined as
\begin{equation}\label{eqn:frequencyFormula}
    N(r) \coloneqq \frac{r}{2} \frac{d}{dr} \ln{H(r)}.
\end{equation}
Briefly, the idea to prove the sharp regularity is as follows.
\begin{itemize}
    \item[(i)] To begin with, in section \ref{subsec:FrequencyFormula} we introduce a suitable truncation of the frequency formula \eqref{eqn:frequencyFormula}, together with some preliminary computations and bounds on error terms.
    
    \item[(ii)] In section \ref{subsec:FrequencyAlmostMonotonicity} we prove that the truncated frequency formula \eqref{eqn:frequencyFormula} is almost-monotone (see Proposition \ref{prop:FrequencyAlmostMonotone} for the precise statement).

    \item[(iii)] Using a blow-up argument to bound from below the frequency formula, thanks to the almost-optimal regularity (Theorem \ref{thm:C1aRegularity}) and the almost-monotonicity (Proposition \ref{prop:FrequencyAlmostMonotone}) we prove that
    \begin{equation}\label{eqn:AlmostSharpBoundH}
        N(r) \ge 3/2, \quad r \in (0, c)
    \end{equation}
    for a constant $c$ independent of $\beta$ and depending possibly on $x_0$, but uniform in a neighborhood of any regular point of the free boundaries. This is the content of section \ref{subsec:FrequencyBoundBelow}.

    \item[(iv)] Thanks to \eqref{eqn:AlmostSharpBoundH} we obtain the sharp bound
    \[
        H(r) \le C r^{3}, \quad  r \in (0, c) ,
    \]
    from which we directly deduce Theorem \ref{thm:SharpRegularity}. This is carried out in section \ref{subsec:SharpRegularity}.
\end{itemize}
\begin{remark}\label{rmk:FrequencyUniformConstantsNeigborhood}
 In the following, we basically fix a regular point $x_0 \in \partial \{ u>0 \} \cap \partial \{ v>0 \}$ and derive results at that point. However, it is clear that by the uniform $C^{1, \beta}$ regularity obtained in Theorem \ref{thm:C1aRegularity}, all such constants are uniform in $B_r(x_0) \cap \partial \{ u>0 \} \cap \partial \{ v>0 \}$, for some $r>0$ sufficiently small.
\end{remark}
Before we begin with the proof, several preliminaries are in order. 

Let $x_0 \in \partial \{ u>0 \} \cap \partial \{ v>0 \}$ be a regular point for the free boundaries. Without loss of generality, suppose $x_0 = 0$ and that the coordinate system is rotated so that the inner normal of the free boundaries at $x_0$ is directed towards $e_N$. 

For notational convenience, let us define the functions
\begin{align}\label{eqn:FrequencyWuWvDef}
\begin{split}
    w_u \coloneqq \frac{1}{\sqrt{\Lambda_u}} \left( u - \sqrt{\Lambda_u} x_N \right), \quad x \in \{ u>0 \} \cap B_1 ,\\
    w_v \coloneqq \frac{1}{\sqrt{\Lambda_v}} \left( v - \sqrt{\Lambda_v} x_N \right), \quad x \in \{ v>0 \} \cap B_1.
\end{split}
\end{align}
Thanks to Theorem \ref{thm:C1aRegularity}, there exist constants $C_{\beta}$ (that satisfy the same properties following \eqref{eqn:AlmostSharpBoundH}, but depend also possibly on $\beta$) such that, for all $0 < \beta < 1/2$ and $r \in (0, 1)$
\begin{align}
       & \| w_u \|_{L^{\infty}\left( \{ u>0 \} \cap B_r \right)} \le C_{\beta} r^{1+\beta}, \quad \| w_v \|_{L^{\infty}\left( \{ v>0 \} \cap B_r \right)} \le C_{\beta} r^{1+\beta}, \label{eqn:SharpRegLinftyBoundW}\\
       & \| \nabla w_u \|_{L^{\infty}\left( \{ u>0 \} \cap B_r \right)} \le C_{\beta} r^{\beta}, \quad \| \nabla w_v \|_{L^{\infty}\left( \{ v>0 \} \cap B_r \right)} \le C_{\beta} r^{\beta}, \label{eqn:SharpRegBoundGradW}\\
       & \left\vert \partial_{e_N} w_u \right\vert \le C_{\beta} r^{2\beta}, \quad \left\vert \partial_{\nu} w_u \right\vert \le C_{\beta} r^{2\beta} \text{ on } \partial \{ u>0 \} \setminus \partial \{ v>0 \} , \label{eqn:SharpRegBoundPartialWuOnePhase}\\
       & \left\vert \partial_{e_N} w_v \right\vert \le C_{\beta} r^{2\beta}, \quad \left\vert \partial_{\nu} w_v \right\vert \le C_{\beta} r^{2\beta} \text{ on } \partial \{ v>0 \} \setminus \partial \{ u>0 \} , \label{eqn:SharpRegBoundPartialWvOnePhase}\\
       & \partial_{e_N} w_u \le C_{\beta} r^{2\beta}, \quad \partial_{\nu} w_u \le C_{\beta} r^{2\beta} \text{ on } \partial \{ u>0 \} \cap \partial \{ v>0 \} , \label{eqn:SharpRegBoundPartialWuTwoPhase}\\
       & \partial_{e_N} w_v \ge - C_{\beta} r^{2\beta}, \quad \partial_{\nu} w_v \ge - C_{\beta} r^{2\beta} \text{ on } \partial \{ u>0 \} \cap \partial \{ v>0 \} , \label{eqn:SharpRegBoundPartialWvTwoPhase} \\
       & \left\vert \Lambda_u \partial_{\nu} w_u + \Lambda_v \partial_{\nu} w_v \right\vert \le C_{\beta} r^{2\beta} \text{ on } \partial \{ u>0 \} \cap \partial \{ v>0 \} \label{eqn:SharpRegBoundPartialSumTwoPhase}, 
\end{align}
where $\nu$ denotes the \emph{inner} unit normal vector to the domains under consideration. Conditions \eqref{eqn:SharpRegBoundPartialWuOnePhase}, \eqref{eqn:SharpRegBoundPartialWvOnePhase}, \eqref{eqn:SharpRegBoundPartialWuTwoPhase},  \eqref{eqn:SharpRegBoundPartialWvTwoPhase} and \eqref{eqn:SharpRegBoundPartialSumTwoPhase} follow from \eqref{eqn:SharpRegBoundGradW} and the overdetermined conditions on the free boundaries.

\begin{remark}
    In what follows, we always implicitly assume that $\partial \{ u >0 \} \cap B_r$ and $\partial \{ v>0 \} \cap B_r$ are sufficiently close to $\{ x_N = 0 \} \cap B_r$, for instance in the sense of Hausdorff distance. Of course, this is not restrictive after a suitable scaling.
\end{remark}

\subsection{The frequency formula}\label{subsec:FrequencyFormula}
In order to prove the almost monotonicity of the frequency function (that will be introduced at the end of this section, see \eqref{eqn:Truncatedrequency3/2Def}), we need to compute the first and second derivatives of the energy function and carefully estimate the error terms. This will allow us to introduce a suitable truncated frequency formula, defined in terms of a suitably modified height function. This is basically the content of the section.

Given the definitions of $w_u$ and $w_v$, we can rewrite the height function $H(r)$ as
\[
H(r) = \frac{1}{r^{N-1}} \int_{\{u>0 \} \cap \partial B_r} \Lambda_u w_u^2 + \frac{1}{r^{N-1}} \int_{\{v>0 \} \cap \partial B_r} \Lambda_v w_v^2 .
\]
By a direct computation
\begin{align*}
    \frac{d}{dr} H(r) = & \frac{2}{r^{N-1}} \int_{\{u>0 \} \cap B_r} \Lambda_u \vert \nabla w_u \vert^2 + \frac{2}{r^{N-1}} \int_{\{v>0 \} \cap B_r} \Lambda_v\vert \nabla w_v \vert^2 + \\
    + & \frac{2}{r^{N-1}} \int_{\partial \{u>0 \} \cap B_r} \Lambda_u  w_u \partial_{\nu} w_u + \frac{2}{r^{N-1}} \int_{\partial \{v>0 \} \cap B_r} \Lambda_v w_v \partial_{\nu} w_v + \\
    + & \frac{1}{r^N} \int_{\partial \{u>0 \} \cap \partial B_r} \Lambda_u  w_u^2 (x \cdot \nu) + \frac{1}{r^N} \int_{\partial \{v>0 \} \cap \partial B_r} \Lambda_v w_v^2 (x \cdot \nu) .
\end{align*}
Let us denote
\begin{align*}
A(r) \coloneqq & \frac{2}{r^{N-1}} \int_{\partial \{u>0 \} \cap B_r} \Lambda_u  w_u \partial_{\nu} w_u + \frac{2}{r^{N-1}} \int_{\partial \{v>0 \} \cap B_r} \Lambda_v w_v \partial_{\nu} w_v , \\
B(r) \coloneqq & \frac{1}{r^N} \int_{\partial \{u>0 \} \cap \partial B_r} \Lambda_u  w_u^2 (x \cdot \nu) + \frac{1}{r^N} \int_{\partial \{v>0 \} \cap \partial B_r} \Lambda_v w_v^2 (x \cdot \nu) .
\end{align*}
Now we proceed estimating $A(r)$ and $B(r)$.

Let us begin with $A(r)$. We rewrite it separating the contributions coming from the one-phase and the two-phase points.
\begin{align*}
A(r) = & \frac{2}{r^{N-1}} \int_{ \left(\partial \{u>0 \} \setminus \partial \{v>0 \} \right) \cap B_r} \Lambda_u  w_u \partial_{\nu} w_u + \frac{2}{r^{N-1}} \int_{ \left(\partial \{v>0 \} \setminus \partial \{u>0 \} \right) \cap B_r} \Lambda_v  w_v \partial_{\nu} w_v + \\
+ & \frac{2}{r^{N-1}} \int_{ \left(\partial \{u>0 \} \cap \partial \{v>0 \} \right) \cap B_r} \left( \Lambda_u  w_u \partial_{\nu} w_u + \Lambda_v  w_v \partial_{\nu} w_v  \right) \coloneqq A_1(r) + A_2(r),
\end{align*}
where $A_1(r)$ and $A_2(r)$ describe the contributions to $A(r)$ of the one-phase and two-phase points respectively. By \eqref{eqn:SharpRegLinftyBoundW}, \eqref{eqn:SharpRegBoundPartialWuOnePhase} and \eqref{eqn:SharpRegBoundPartialWvOnePhase} we have that
\[
\vert A_1(r) \vert \le C_{\beta} r^{1+3\beta}.
\]
To estimate $A_2(r)$, notice that on $\partial \{u>0 \} \cap \partial \{v>0 \}$ we have $w_u = w_v = -x_N$. Hence, by \eqref{eqn:SharpRegLinftyBoundW} and \eqref{eqn:SharpRegBoundPartialSumTwoPhase}
\[
\vert A_2(r) \vert \le C_{\beta} r^{1+3\beta}.
\]
Consequently,
\begin{equation}\label{eqn:FrequencyA(r)Bound}
    \vert A(r) \vert \le C_{\beta} r^{1+3\beta}.
\end{equation}
Now we estimate $B(r)$. To this aim, it is sufficient to notice that, by \eqref{eqn:SharpRegLinftyBoundW}
\begin{equation}\label{eqn:FrequencyB(r)Bound}
    \vert B(r) \vert \le C_{\beta} r^{1+3\beta}.
\end{equation}
As in \cite{ChangLaraSavin:BoundaryRegularityOnePhase}, in order to have nice formulae for the derivatives, we introduce the modified height function
\[
\tilde{H}(r) \coloneqq H(r) - \int_0^r (A(\rho) + B(\rho)).
\]
Notice that, by \eqref{eqn:FrequencyA(r)Bound} and \eqref{eqn:FrequencyB(r)Bound}
\begin{equation}\label{eqn:HTildeHDiffBound}
\vert \tilde{H}(r) - H(r) \vert \le C_{\beta} r^{2+3\beta} ,
\end{equation}
and by definition
\begin{equation}\label{eqn:TildeHFirstDerivative}
    \frac{d}{dr} \tilde{H}(r) = \frac{2}{r^{N-1}} \int_{\{u>0 \} \cap B_r} \Lambda_u \vert \nabla w_u \vert^2 + \frac{2}{r^{N-1}} \int_{\{v>0 \} \cap B_r} \Lambda_v\vert \nabla w_v \vert^2 .
\end{equation}
Now we give some preliminary properties for the second derivative of $\tilde{H}(r)$, that will be used in the following sections (see Remark \ref{rmk:FrequencyUniformConstantsNeigborhood} for the dependence of the constants on $x_0$).
\begin{lemma}\label{lemma:FrequencyBound2ndDeriv}
    For all $0 < \beta < 1/2$ there exists a constant $C_{\beta}$ (possibly depending also on $\beta$) such that, for all $r \in (0, 1)$
    \begin{equation}\label{eqn:FrequencyBound2ndDeriv}
        \left \vert\frac{d^2}{d r^2} \tilde{H}(r) + \frac{1}{r} \frac{d}{dr} \tilde{H}(r) - \frac{4}{r^{N-1}} \int_{\{u>0\} \cap \partial B_r} \Lambda_u \left( \partial_r w_u \right)^2 - \frac{4}{r^{N-1}} \int_{\{v>0\} \cap \partial B_r} \Lambda_v \left( \partial_r w_v \right)^2 \right\vert \le C_{\beta} r^{3\beta} .
    \end{equation}
\end{lemma}
\begin{proof}
By a direct computation,
\[
\frac{d^2}{d r^2} \tilde{H}(r) + \frac{n-1}{r} \frac{d}{dr} \tilde{H}(r) = \frac{2}{r^{N-1}} \int_{ \{u>0\} \cap \partial B_r } \Lambda_u \vert \nabla w_u \vert^2 + \frac{2}{r^{N-1}} \int_{ \{v>0\} \cap \partial B_r } \Lambda_v \vert \nabla w_v \vert^2 .
\]
Computing $\frac{d}{dr}\tilde{H}(r)$ using the following identity
\[
(N-2) \vert \nabla f \vert ^2 - \left( \nabla f \cdot x \right) \Delta f = \diverg \left( \vert \nabla f \vert^2 x - \left( \nabla f \cdot x \right) \nabla f \right) 
\]
and separating the contributions coming from $\partial B_r$ and the free boundaries, we obtain
\begin{align*}
    \frac{d^2}{d r^2} \tilde{H}(r) + \frac{1}{r} \frac{d}{dr} \tilde{H}(r) = & \frac{4}{r^{N-1}} \int_{\{u>0\} \cap \partial B_r} \Lambda_u \left( \partial_r w_u \right)^2 + \frac{4}{r^{N-1}} \int_{\{v>0\} \cap \partial B_r} \Lambda_v \left( \partial_r w_v \right)^2 + \\
    + & \frac{2}{r^N} \int_{ \partial \{u>0\} \cap B_r} \Lambda_u \left( \vert \nabla w_u \vert^2 (x \cdot \nu) - (\nabla w_u \cdot x) \partial_{\nu} w_u  \right) + \\
    + & \frac{2}{r^N} \int_{ \partial \{v>0\} \cap B_r} \Lambda_v \left( \vert \nabla w_v \vert^2 (x \cdot \nu) - (\nabla w_v \cdot x) \partial_{\nu} w_v  \right) .
\end{align*}
In order to estimate the order of the terms in the right hand side integrated on the free boundaries, we separate the contributions of the one-phase and two-phase points. To this aim, we define
\begin{align}\label{eqn:FrequencyCompute2ndDeriv}
\begin{split}
    C(r) \coloneqq & \frac{2}{r^N} \int_{ \left( \partial \{u>0\} \setminus \partial \{v>0\} \right)\cap B_r} \Lambda_u \left( \vert \nabla w_u \vert^2 (x \cdot \nu) - (\nabla w_u \cdot x) \partial_{\nu} w_u  \right) + \\
    + & \frac{2}{r^N} \int_{ \left( \partial \{v>0\} \setminus \partial \{u>0\} \right)\cap B_r} \Lambda_v \left( \vert \nabla w_v \vert^2 (x \cdot \nu) - (\nabla w_v \cdot x) \partial_{\nu} w_v  \right) , \\
    D(r) \coloneqq & \frac{2}{r^N} \int_{ \left( \partial \{u>0\} \cap \partial \{v>0\} \right)\cap B_r} \left( \Lambda_u  \vert \nabla w_u \vert^2 + \Lambda_v  \vert \nabla w_v \vert^2 \right) (x \cdot \nu) , \\
    E(r) \coloneqq & \frac{2}{r^N} \int_{ \left( \partial \{u>0\} \cap \partial \{v>0\} \right)\cap B_r} \left( \Lambda_u (\nabla w_u \cdot x) \partial_{\nu} w_u + \Lambda_v (\nabla w_v \cdot x) \partial_{\nu} w_v \right) .
\end{split}
\end{align}
Concerning $C(r)$, we see from \eqref{eqn:SharpRegBoundGradW}, \eqref{eqn:SharpRegBoundPartialWuOnePhase} and \eqref{eqn:SharpRegBoundPartialWvOnePhase} that
\begin{equation}\label{eqn:FrequencyC(r)Bound}
\vert C(r) \vert \le C_{\beta} r^{3\beta}.
\end{equation}
In a similar way, we see from \eqref{eqn:SharpRegLinftyBoundW} and \eqref{eqn:SharpRegBoundGradW} that
\begin{equation}\label{eqn:FrequencyD(r)Bound}
\vert D(r) \vert \le C_{\beta} r^{3\beta}.
\end{equation}
In order to estimate $E(r)$, we rewrite the integrand as
\begin{align*}
    & \Lambda_u (\nabla w_u \cdot x) \partial_{\nu} w_u + \Lambda_v (\nabla w_v \cdot x) \partial_{\nu} w_v = \\
    = &(\nabla w_u \cdot x) \left(\Lambda_u \partial_{\nu} w_u + \Lambda_v \partial_{\nu} w_v \right) + \Lambda_v \left( (\nabla w_v - \nabla w_u) \cdot x \right) \partial_{\nu} w_v = \\
    = & (\nabla w_u \cdot x) \left(\Lambda_u \partial_{\nu} w_u + \Lambda_v \partial_{\nu} w_v \right) + \Lambda_v \left( \frac{\partial_{\nu} u}{\sqrt{\Lambda_u}} - \frac{\partial_{\nu} v}{\sqrt{\Lambda_v}} \right) (x \cdot \nu) \partial_{\nu} w_v
\end{align*}
Using \eqref{eqn:SharpRegBoundGradW} and \eqref{eqn:SharpRegBoundPartialSumTwoPhase} to estimate the first addend, \eqref{eqn:SharpRegLinftyBoundW}, \eqref{eqn:SharpRegBoundGradW} and the definitions of $w_u$ and $w_v$ to estimate the second addend, we obtain that
\begin{equation}\label{eqn:FrequencyE(r)Bound}
\vert E(r) \vert \le C_{\beta} r^{3\beta}.
\end{equation}
Combining \eqref{eqn:FrequencyCompute2ndDeriv}, \eqref{eqn:FrequencyC(r)Bound}, \eqref{eqn:FrequencyD(r)Bound} and \eqref{eqn:FrequencyE(r)Bound} we deduce \eqref{eqn:FrequencyBound2ndDeriv}.
\end{proof}
To conclude this section, we introduce the frequency formula we will be using in the following. Let $0 < \sigma < 1/4$ denote a small positive real number, and introduce the following truncated frequency function
\begin{equation}\label{eqn:Truncatedrequency3/2Def}
    \tilde{N}(r) \coloneqq \frac{r}{2} \frac{d}{dr} \ln{\max \left(\tilde{H}(r), r^{3+\sigma} \right) }.
\end{equation}
We proceed in the following section proving an almost-monotonicity result for $\tilde{N}(r)$.

\subsection{Compactness of blow-up sequences}\label{subsec:FrequencyCompactnessBlowUp}
The main technical tool that we use to deal with the sharp regularity, is the following (see also Remark \ref{rmk:FrequencyUniformConstantsNeigborhood} for the dependence of the constants on $x_0$) 
\begin{lemma}\label{lemma:FrequencyCompactnessBlowUp}
    Let $x_0 \in \partial \Omega_u \cap \partial \Omega_v$ be a regular point for the free boundaries, $r_n \to 0$ a sequence of positive numbers, $w_u$, $w_v$ as in \eqref{eqn:FrequencyWuWvDef} and define the rescaled functions
    \[
    w_{u, n}(x) \coloneqq \frac{w_u(r_n)(x)}{\tilde{H}(r_n)^{1/2}} , \quad w_{v,n}(x) \coloneqq \frac{w_v(r_n)(x)}{\tilde{H}(r_n)^{1/2}} , \quad u_n(x) \coloneqq \frac{u(r_n)(x)}{\tilde{H}(r_n)^{1/2}} , \quad v_n(x) \coloneqq \frac{v(r_n)(x)}{\tilde{H}(r_n)^{1/2}}
    \]
    and the rescaled domains
    \[
    \Omega_{u, n} \coloneqq \left\{ x \in \R^N ; r_n (x + x_0) \in \Omega_u \cap B_{r_n}(x_0)\right\} \quad \text{and} \quad \Omega_{v, n} \coloneqq \left\{ x \in \R^N ; r_n (x + x_0) \in \Omega_v \cap B_{r_n}(x_0)\right\} .
    \]
    Suppose that, for some $0<\sigma<1/4$ and constant $K>0$
    \begin{equation}\label{eqn:FrequencyCompactnessHp}
            \tilde{H}(r_n) \ge r_n^{3+\sigma} \quad \text{and} \quad
            r_n\frac{d}{dr}\ln{\tilde{H}(r_n)} = r_n \frac{\frac{d}{dr} \tilde{H}(r_n)}{\tilde{H}(r_n)} \le K .
    \end{equation}
    Then, any fixed $0 < \beta < 1/2$ sufficiently close to $1/2$ and $A \Subset B_1$ there exists a constant $C_{A, \beta}>0$ (satisfying the same properties following \eqref{eqn:AlmostSharpBoundH}, but possibly depending also on $\beta$ and $A$) such that
    \begin{itemize}
        \item[(i)] 
        \[
        \| w_{u, n} \|_{L^{\infty}\left( \overline{\Omega_{u, n} \cap A} \right)} + \| w_{v, n} \|_{L^{\infty}\left( \overline{\Omega_{v, n} \cap A} \right)} \le C_{A, \beta} (1+K) ,
        \]

        \item[(ii)] For $n$ sufficiently large
        \[
        \| w_{u, n} \|_{C^{1, \beta}\left( \overline{\Omega_{u, n} \cap A} \right)} + \| w_{v, n} \|_{C^{1, \beta}\left( \overline{\Omega_{v, n} \cap A} \right)} \le C_{A, \beta} (1+K) .
        \]
    \end{itemize}
\end{lemma}
\begin{remark}\label{rmk:PrecompactH1/2Tor^3/2}
    An analogous proof works if $\tilde{H}(r_n)$ is substituted by $r_n^{3/2}$ (with the second hypothesis still remaining $\tilde{N}(r_n) \le K$), as long as we know that $\tilde{H}(r) \le C r^{3/2}$. 
\end{remark}
\begin{proof}
Without loss of generality, we can assume $x_0 = 0$.

Point $(ii)$ follows immediately from point (i) and Remark \ref{rmk:ImprovFlatEstimateLinearizedFunctions}, as long as the condition
\[
\frac{\tilde{H}(r_n)^{1/2}}{r_n} \le \varepsilon_0,
\]
which is satisfied for $n$ sufficiently large, thanks to \eqref{eqn:SharpRegLinftyBoundW}.

Now we proceed to prove point (i). To this aim we bound $w_{u, n}$ and $w_{v, n}$ from above and below, using harmonic functions solving mixed Dicichlet/Neumann problems. For this reason, we divide the proof in four steps, depending on the bound we want to achieve. We only give the proof of the main steps and sketch the similar parts.
\begin{itemize}
\item[a)] We begin bounding $w_{v, n}^+$ from above. To this aim, introduce the function $h_{v, n}(x) \in H^1\left( \Omega_{v, n} \cap B_1\right)$, weak solution of
\[
\begin{cases}
    \Delta h_{v, n} = 0 & \text{in } \Omega_{v, n} \cap B_1 , \\
    h_{v, n}= w_{v, n}^+ & \text{on } \Omega_{v, n} \cap \partial B_1, \\
    \partial_{\nu} h_{v, n} = 0 & \text{on } \partial \Omega_{v, n} \cap B_1.
\end{cases}
\]
In particular,
\[
\int_{\Omega_{v, n} \cap B_1 } \vert \nabla h_{v, n} \vert^2 \le \int_{\Omega_{v, n} \cap B_1 } \vert \nabla w_{v, n}^+ \vert^2
\]
Moreover, notice that by elliptic regularity (see e.g. \cite[Theorem 6.13]{GilbargTrudinger:Elliptic98}), a flattening of the boundary and a reflection, thanks to the exterior cone condition we have $h_{v, n} \in C^0\left(\overline{\Omega_{v, n} \cap B_1} \right) \cap C^{2, \alpha}_{loc}\left(\Omega_{v, n} \cap B_1 \right)$ for some $0 < \alpha < 1$. 

By \eqref{eqn:SharpRegBoundPartialWvTwoPhase} and the (e.g. viscous) comparison principle, we have
\begin{equation}\label{eqn:FrequencyCompactnessProof0}
    h_{v, n}(x) \ge 0 \quad \text{and} \quad w_{v, n}^+(x) \le h_{v, n}(x) - C_{\beta} \frac{r_n^{1+2 \beta}}{\tilde{H}(r_n)^{1/2}} (x_n - 1) \quad \text{in } \overline{\Omega_{v, n} \cap B_1}.
\end{equation}
Now, by (interior, after reflection) elliptic regularity we have that
\[
\| h_{v, n} \|_{L^{\infty}\left( \overline{\Omega_{v, n} \cap A} \right)} \le C_{A} \| h_{v, n} \|_{L^2\left( \Omega_{v, n} \cap B_1 \right)}.
\]
Moreover, integrating over lines directed as $e_N$, starting from $\partial B_1 \cap \Omega_{v, n}$, we have
\begin{equation}\label{eqn:FrequencyCompactnessProof1}
    \int_{\Omega_{v, n} \cap B_1} h_n^2 \le C \int_{\Omega_{v, n} \cap \partial B_1} \left( w_{v, n}^+ \right)^2 + \int_{\Omega_{v, n} \cap B_1} \left\vert \nabla w_{v, n}^+ \right\vert^2 \le C (1+K),
\end{equation}
where in the last inequality we have used \eqref{eqn:FrequencyCompactnessHp}. Using again \eqref{eqn:FrequencyCompactnessHp} we also have that
\begin{equation}\label{eqn:FrequencyCompactnessProof2}
    \frac{r_n^{1+2 \beta}}{\tilde{H}(r_n)^{1/2}} \le r_n^{1/4},
\end{equation}
so that, combining \eqref{eqn:FrequencyCompactnessProof0}, \eqref{eqn:FrequencyCompactnessProof1} and \eqref{eqn:FrequencyCompactnessProof2} and using once again \eqref{eqn:FrequencyCompactnessHp} we obtain
\begin{equation}\label{eqn:FrequencyCompactnessWvBdAbove}
    \| w_{v, n}^+ \|_{L^{\infty}\left( \overline{\Omega_{v, n} \cap A} \right)} \le C_{A, \beta} (1+K).
\end{equation}

\item[b)] Next, we bound $w_{u, n}$ from below. Since this is equivalent to bound $w_{u, n}^-$ from above, proceeding exactly as in point a), but using \eqref{eqn:SharpRegBoundPartialWuTwoPhase} in place of \eqref{eqn:SharpRegBoundPartialWvTwoPhase}, we obtain
\begin{equation}\label{eqn:FrequencyCompactnessWuBdBelow}
    \| w_{u, n}^- \|_{L^{\infty}\left( \overline{\Omega_{u, n} \cap A} \right)} \le C_{A, \beta} (1+K).
\end{equation}
We omit the proof.

\item[c)] Let us bound $w_{v, n}$ from below, or better $w_{v, n}^-$ from above. To this aim, we exploit the bound from above \eqref{eqn:FrequencyCompactnessWuBdBelow} on $w_{u, n}^-$. Indeed, notice that $w_{u, n} = w_{v, n}$ on $\partial \{ u_n > 0 \} \cap \partial \{ v_n > 0 \}$, so that $w_{v, n}^- \le C_{A, \beta} (1+K)$ on such set.

Now, defining (with a little abuse of notation) $h_{v, n}(x) \in H^1\left( \Omega_{v, n} \cap B_1\right)$ as the weak solution of
\[
\begin{cases}
    \Delta h_{v, n} = 0 & \text{in } \Omega_{v, n} \cap B_1 , \\
    h_{v, n} = w_{v, n}^- & \text{on } \Omega_{v, n} \cap \partial B_1, \\
    \partial_{\nu} h_{v, n} = 0 & \text{on } \partial \Omega_{v, n} \cap B_1,
\end{cases}
\]
by the previous observation, \eqref{eqn:SharpRegBoundPartialWvOnePhase} and the comparison principle we have
\[
    h_{v, n}(x) \ge 0 \quad \text{and} \quad w_{v, n}^-(x) \le h_{v, n}(x) + C_{A, \beta} (1+K) - C_{\beta} \frac{r_n^{1+2 \beta}}{\tilde{H}(r_n)^{1/2}} (x_n -1 ) \quad \text{in } \overline{\Omega_{v, n} \cap B_1}.
\]
Hence, proceeding as in point a) we deduce
\begin{equation}\label{eqn:FrequencyCompactnessWvBdBelow}
    \| w_{v, n}^- \|_{L^{\infty}\left( \overline{\Omega_{v, n} \cap A} \right)} \le C_{A, \beta} (1+K).
\end{equation}

\item[d)] We are left to prove a bound from above for $w_{u, n}^+$. Similarly as in the previous point, we can use \eqref{eqn:FrequencyCompactnessWvBdAbove} to bound $w_{u, n}^+$ from above at the two-phase points. Combining this with \eqref{eqn:SharpRegBoundPartialWuOnePhase} we obtain
\begin{equation}\label{eqn:FrequencyCompactnessWuBdAbove}
    \| w_{u, n}^+ \|_{L^{\infty}\left( \overline{\Omega_{v, n} \cap A} \right)} \le C_{A, \beta} (1+K).
\end{equation}
\end{itemize}
Combining \eqref{eqn:FrequencyCompactnessWvBdAbove}, \eqref{eqn:FrequencyCompactnessWuBdBelow}, \eqref{eqn:FrequencyCompactnessWvBdBelow} and \eqref{eqn:FrequencyCompactnessWuBdAbove} conludes the proof of point (i).
\end{proof}

\subsection{Almost-monotinicity}\label{subsec:FrequencyAlmostMonotonicity}
The main result of this section is contained in the following (we refer to Remark \ref{rmk:FrequencyUniformConstantsNeigborhood} for the dependence of the constants on the point $x_0$)
\begin{proposition}\label{prop:FrequencyAlmostMonotone}
There exist constants $c, C>0$ such that the function
\[
(1+Cr^{\sigma}) \tilde{N}(r), \quad r \in (0, c)
\]
is non-decreasing in $r$.
\end{proposition}
\begin{remark}\label{rmk:FrequencyMonotoneRegimeToProve}
 We only need to analyze the case $\tilde{H}(r) \ge r^{3+\sigma}$ since in the regime $\tilde{H}(r) \le r^{3+\sigma}$ a direct computation gives $\frac{d}{dr} \tilde{H}(r) = 0$.   
\end{remark}
It will be clear from the proof of Proposition \ref{prop:FrequencyAlmostMonotone} that we need to estimate from below the quantity $\frac{d}{dr}\tilde{N}(r)$. This is the content of the following (Remark \ref{rmk:FrequencyUniformConstantsNeigborhood} deals with the dependence of the constants on $x_0$)
\begin{lemma}\label{lemma:DerivativeTildeHBoundBelow}
    Suppose that
    \begin{equation}\label{eqn:DerivativeTildeHBdBelowHp}
    \tilde{H}(r) \ge r^{3+\sigma}.     
    \end{equation}
    Then, there exist constants $c, C >0$ such that (see also \eqref{eqn:TildeHFirstDerivative})
    \begin{equation*}\label{eqn:DerivativeTildeHBdBelowTh}
    \frac{d}{dr}\tilde{H}(r) \ge C r^{2 + \sigma} \quad \text{if } r \in (0, c).      
    \end{equation*}
\end{lemma}
\begin{proof}
We proceed by contradiction. Hence, suppose that there exist radii $r_n \to 0$ and positive constants $C_n \to 0$ such that
\begin{equation}\label{eqn:DerivtildeHBoundBelowContradict}
    \frac{d}{dr}\tilde{H}(r_n) < C_n r_n^{2 + \sigma}.
\end{equation}
Thanks to \eqref{eqn:DerivativeTildeHBdBelowHp} and \eqref{eqn:DerivtildeHBoundBelowContradict} we are under the hypotheses of Lemma \ref{lemma:FrequencyCompactnessBlowUp}, with $K = K_n = C_n \to 0$.

Up to a small deformation (Lipschitz continuous near $\partial B_1$ and smooth elsewhere), we can assume that for $n$ large enough,
\[
\{ w_{u, n} > 0 \} \cap B_1 = \{ w_{v, n} > 0 \} \cap B_1 = B_1 \cap \{ x_N > 0 \} \coloneqq B_1^+
\]
that (see \eqref{eqn:FrequencyCompactnessProof1})
\[
\| w_{u, n} \|_{H^1 \left( \Omega_{u, n} \cap B_1 \right)} + \| w_{v, n} \|_{H^1 \left( \Omega_{v, n} \cap B_1 \right)} \le C(1+C_n).
\]
and that conditions \eqref{eqn:DerivativeTildeHBdBelowHp}, \eqref{eqn:DerivtildeHBoundBelowContradict} still hold true.

Let $\overline{r} < 1$ to be fixed later. Up to a subsequence that we do not relabel, there exist two functions $w_{u, \infty}, w_{v, \infty} \in H^1\left( B_1^+ \right) \cap C^{1, \beta}\left( \overline{B_{\overline{r}}^+} \right)$ such that
\begin{align*}
    & w_{u, n} \rightharpoonup w_{u, \infty} \quad \text{and} \quad w_{v, n} \rightharpoonup w_{v, \infty} \quad \text{weakly in } H^1\left( B_1^+ \right) , \\
    & w_{u, n} \to w_{u, \infty} \quad \text{and} \quad w_{v, n} \to w_{v, \infty} \quad \text{strongly in } L^2\left( B_1^+ \right) , \\
    & w_{u, n} \to w_{u, \infty} \quad \text{and} \quad w_{v, n} \to w_{v, \infty} \quad \text{strongly in } C^{1, \beta}\left( \overline{B_{\overline{r}}^+} \right).
\end{align*}
By \eqref{eqn:DerivativeTildeHBdBelowHp} and the contradiction hypothesis \eqref{eqn:DerivtildeHBoundBelowContradict}, we also have
\begin{equation}\label{eqn:AlmostMonotoneGradWuWvL2To0}
    \| \nabla w_{u, n} \|_{L^2\left( B_1^+ \right)} \to 0 \quad \text{and} \quad \| \nabla w_{v, n} \|_{L^2\left( B_1^+ \right)} \to 0,
\end{equation}
so that $w_{u, \infty}$ and $w_{v, \infty}$ are constant over $B_1^+$. However, since $w_{u, n}(0) = w_{v, n}(0) = 0$, by the strong uniform convergence in $\overline{B_{\overline{r}}^+}$ we necessarily have
\[
    w_{u, \infty} = w_{v, \infty} \equiv 0 \text{ in } B_1^+,
\]
and in particular
\begin{equation}\label{eqn:AlmostMonotoneWuWvUnifTo0}
    \| w_{u, n} \|_{C^{1, \beta}\left( \overline{B_r^+} \right)} \to 0 \quad \text{and} \quad \| w_{v, n} \|_{C^{1, \beta}\left( \overline{B_r^+} \right)} \to 0.
\end{equation}
Morever, by the trace inequality, \eqref{eqn:AlmostMonotoneGradWuWvL2To0} and the  normalization chosen (i.e. $L^2$-norm of the trace unitary), there exists a constant $K > 0$ depending only on the dimension such that
\begin{equation}\label{eqn:AlmostMonotoneBdL2Below}
    \| w_{u, n} \|_{L^2\left( B_1^+ \right)} + \| w_{v, n} \|_{L^2\left( B_1^+ \right)} \ge K,
\end{equation}
Now, we combine \eqref{eqn:AlmostMonotoneGradWuWvL2To0}, \eqref{eqn:AlmostMonotoneWuWvUnifTo0} and \eqref{eqn:AlmostMonotoneBdL2Below} to reach a contradiction.

Let $h>0$ be a small constant, to be chosen later. Then, integrating on radii starting from $\partial B_{\overline{r}}$, we can write

By a Poincaré inequality argument we have
\begin{align*}
    & K + o(1) \le \int_{B_1^+} \Lambda_u (\vert w_{u, n} \vert - h)_+^2 + \Lambda_v (\vert w_{v, n} \vert - h)_+^2 = \\
    & = \int_{S^{N-1}} d\sigma \int_{\overline{r}}^1 \left[ \Lambda_u (\vert w_{u, n} \vert - h)_+^2 + \Lambda_v (\vert w_{v, n} \vert - h)_+^2 \right] \rho^{N-1} d\rho \le \\
    & \le \frac{(1- \overline{r})^2}{\overline{r}^{N-1}} \left( \| \nabla w_{u, n} \|_{L^2\left( B_1^+ \right)} + \| \nabla w_{v, n} \|_{L^2\left( B_1^+ \right)} \right) ,
\end{align*}
where the quantity $o(1)$ is intended as $n \to +\infty$, $\overline{r} \to 1^+$ and $h \to +^+$. This is possible thanks to \eqref{eqn:AlmostMonotoneWuWvUnifTo0} and \eqref{eqn:AlmostMonotoneBdL2Below}. Hence, for some fixed $\overline{r}$ sufficiently close to $1$ and $h$ sufficiently close to 0 , for $n$ accordingly large enough we have that
\[
\left( \| \nabla w_{u, n} \|_{L^2\left( B_1^+ \right)} + \| \nabla w_{v, n} \|_{L^2\left( B_1^+ \right)} \right) \ge K \frac{\overline{r}^{N-1}}{(1- \overline{r})^2} + o(1) \ge 1/2 ,
\]
which contradicts \eqref{eqn:AlmostMonotoneGradWuWvL2To0}. The proof is concluded.
\end{proof}
We are ready to give the proof of Proposition \ref{prop:FrequencyAlmostMonotone}.
\begin{proof}[Proof of Proposition \ref{prop:FrequencyAlmostMonotone}.]
By Remark \ref{rmk:FrequencyMonotoneRegimeToProve} we only need to study the regime $\tilde{H}(r) \ge r^{3+\sigma}$.

Using \eqref{eqn:FrequencyBound2ndDeriv} we have
\begin{align}\label{eqn:FrequencyAlmostMonotoneToEstimate}
    \begin{split}
        & \frac{\frac{d}{dr}\tilde{N}(r)}{\tilde{N}(r)} = \frac{1}{r} + \frac{\frac{d^2}{dr^2}\tilde{N}(r)}{\frac{d}{dr}\tilde{N}(r)} - \frac{\frac{d}{dr}\tilde{N}(r)}{\tilde{N}(r)} \ge \\
        \ge & \frac{1}{\frac{d}{dr}\tilde{H}(r)} \left( \frac{4}{r^{N-1}} \int_{\{u>0\} \cap \partial B_r} \Lambda_u \left( \partial_r w_u \right)^2 + \frac{4}{r^{N-1}} \int_{\{v>0\} \cap \partial B_r} \Lambda_v \left( \partial_r w_v \right)^2 - C_{\beta} r^{3\beta} \right) - \frac{\frac{d}{dr}\tilde{H}(r)}{\tilde{H}(r)}.
    \end{split}
    \end{align}
We need to bound both addends. To begin with, let us denote
\begin{align*}
& F(r) \coloneqq \frac{4}{r^{N-1}} \int_{\{u>0\} \cap \partial B_r} \Lambda_u \left( \partial_r w_u \right)^2 + \frac{4}{r^{N-1}} \int_{\{v>0\} \cap \partial B_r} \Lambda_v \left( \partial_r w_v \right)^2 , \\
& I(r) \coloneqq \frac{2}{r^{N-1}} \int_{\{u>0 \} \cap \partial B_r} \Lambda_u w_u \partial_r w_u + \frac{2}{r^{N-1}} \int_{\{v>0 \} \cap \partial B_r} \Lambda_v w_v \partial_r w_v, \\    
& K(r) \coloneqq \frac{2}{r^{N-1}} \int_{\partial \{u>0 \} \cap B_r} \Lambda_u w_u \partial_{\nu} w_u + \frac{2}{r^{N-1}} \int_{\partial \{v>0 \} \cap B_r} \Lambda_v w_v \partial_{\nu} w_v .
\end{align*}
Notice that by Lemma \ref{lemma:DerivativeTildeHBoundBelow} we have $I(r) \ge C r^{2+\sigma}$, while by \eqref{eqn:FrequencyA(r)Bound} we have $\vert K(r) \vert \le C_{\beta} r^{1+3\beta}$. Hence, for some fixed $0<\beta<1/2$ but sufficiently close to $1/2$ and taking into account \eqref{eqn:TildeHFirstDerivative}, we have
\begin{equation}\label{eqn:FrequencyAlmostMonI(r)expansion}
    \frac{1}{\frac{d}{dr}\tilde{N}(r)} = \frac{1}{I(r)+K(r)} = \frac{1}{I(r)} (1 + O(r^{3\beta - 1 - \sigma})).
\end{equation}
Moreover, using that $\tilde{H}(r) \ge r^{3+\sigma}$ together with \eqref{eqn:SharpRegBoundGradW} and a Cauchy-Schwarz inequality.
\begin{equation}\label{eqn:FrequencyAlmostMonF(r)/I(r)Bd}
    \frac{F(r)}{I(r)} \ge \left( \frac{\frac{1}{2} F(r)}{\tilde{H}(r)} \right)^{1/2} = O\left( r^{\beta - 3/2 - \sigma/2} \right).
\end{equation}
Now, combining Lemma \ref{lemma:DerivativeTildeHBoundBelow} with \eqref{eqn:FrequencyAlmostMonI(r)expansion} and \eqref{eqn:FrequencyAlmostMonF(r)/I(r)Bd} we estimate
\begin{align}\label{eqn:FrequencyAlmostMonotoneToProve1}
\begin{split}
    & \frac{(F(r) - C_{\beta} r^{3\beta})}{\frac{d}{dr}\tilde{N}(r)} \ge \frac{ F(r)}{I(r)} (1 + O(r^{3\beta - 2 - \sigma})) - C_{\beta} r^{3\beta - \sigma - 2} \ge \\
    &\ge \left( \frac{\frac{1}{2} F(r)}{\tilde{H}(r)} \right)^{1/2} - C_{\beta} r^{4 \beta - 3/2 - 3/2 \sigma - 1} - C_{\beta} r^{3\beta - \sigma - 2} \ge \left( \frac{\frac{1}{2} F(r)}{\tilde{H}(r)} \right)^{1/2} - C_{\beta} r^{\sigma - 1}
\end{split}
\end{align}
if $0<\beta<1/2$ is fixed sufficiently close to $1/2$.

Now, using that $\tilde{H}(r) \ge r^{3+\sigma}$, \eqref{eqn:FrequencyA(r)Bound}, \eqref{eqn:SharpRegBoundGradW} and a Cauchy-Schwarz inequality we have
\begin{align}\label{eqn:FrequencyAlmostMonotoneToProve2}
\begin{split}
    & \frac{\frac{d}{dr}\tilde{H}(r)}{\tilde{H}(r)} = \frac{1}{\tilde{H}(r)} \left( \frac{2}{r^{N-1}} \int_{\{u>0 \} \cap \partial B_r} \Lambda_u w_u \partial_r w_u + \frac{2}{r^{N-1}} \int_{\{v>0 \} \cap \partial B_r} \Lambda_v w_v \partial_r w_v  \right) + \\
    & + \frac{1}{\tilde{H}(r)} \left( \frac{2}{r^{N-1}} \int_{\partial \{u>0 \} \cap B_r} \Lambda_u w_u \partial_{\nu} w_u + \frac{2}{r^{N-1}} \int_{\partial \{v>0 \} \cap B_r} \Lambda_v w_v \partial_{\nu} w_v  \right) \le \\
    & \le \left( \frac{\frac{1}{2} F(r)}{\tilde{H}(r)} \right)^{1/2} + C_{\beta} r^{3\beta - \sigma - 2} \le \left( \frac{\frac{1}{2} F(r)}{\tilde{H}(r)} \right)^{1/2}+C_{\beta} r^{\sigma - 1}
\end{split}
\end{align}
if $0<\beta<1/2$ is fixed sufficiently close to $1/2$.

Combining \eqref{eqn:FrequencyAlmostMonotoneToEstimate}, \eqref{eqn:FrequencyAlmostMonotoneToProve1} and \eqref{eqn:FrequencyAlmostMonotoneToProve2}, for some fixed $0 < \beta < 1/2$ sufficiently close to $1/2$ we get that
\[
\frac{d}{dr}\tilde{N}(r) + C_{\beta} r^{\sigma - 1} \tilde{N}(r) \ge 0
\]
Now the thesis follows for some $c>0$ depending only on the dimension $N$, once $0 < \beta < 1/2$ is fixed sufficiently close to $1/2$.
\end{proof}

\subsection{Bound from below}\label{subsec:FrequencyBoundBelow}
Using Lemma \ref{lemma:FrequencyCompactnessBlowUp} we can prove that blow-up sequences (defined as in the lemma) are precompact, and that their limits are nontrivial homogeneous solutions of the one-sided two membranes problem \eqref{eqn:ViscousOneSidedTwoMembranePb}. As a consequence, thanks to Lemma \ref{prop:FrequencyAlmostMonotone} we deduce a bound from below for the truncated frequency function $\tilde{N}(r)$.

Before we introduce the main result of this section, we recall that (as already remarked in the proof of Lemma \ref{lemma:DerivativeTildeHBoundBelow}) under the hypotheses of Lemma \ref{lemma:FrequencyCompactnessBlowUp}, up to an infinitesimal deformation (of higher order) which is Lipschitz continuous near $\partial B_1$ and smooth elsewhere, we can assume that for large $n$,
\begin{align}\label{eqn:FrequancyBdBelowHp2}
\begin{split}
& \{ w_{u, n} > 0 \} \cap B_1 = \{ w_{v, n} > 0 \} \cap B_1 = B_1 \cap \{ x_N > 0 \} \coloneqq B_1^+ \\
& \qquad \qquad \qquad \qquad \qquad \qquad \text{and} \\
& \text{conditions \eqref{eqn:FrequencyCompactnessHp} still hold true (up to an infinitesimal error).}
\end{split}
\end{align}
After these preliminaries, we are ready to state the following
\begin{lemma}\label{lemma:FrequencyBdFromBelow}
    Under the hypotheses of Lemma \ref{lemma:FrequencyCompactnessBlowUp} together with \eqref{eqn:FrequancyBdBelowHp2}, there exist two functions 
    \[
    w_{u, \infty} , \, w_{v, \infty} \in H^1\left( B_1^+ \right) \cap C^{1, 1/2}\left( \overline{B_1^+ \cap A} \right) \text{ for all } A \Subset B_1
    \]
    such that
    \begin{itemize}
        \item[(i)] $w_{u, n} \to w_{u, \infty}$ and $w_{v, n} \to w_{v, \infty}$ weakly in $H^1\left( B_1^+ \right)$,

        \item[(ii)] up to a subsequence (not relabeled) $w_{u, n} \to w_{u, \infty}$ and $w_{v, n} \to w_{v, \infty}$ in $C^{1, \beta}\left( \overline{B_1^+ \cap A} \right)$, for all $A \Subset B_1$ and $0 < \beta < 1/2$,
    
        \item[(iii)] at least one function between $w_{u, \infty}$ and $w_{v, \infty}$ is nontrivial in $B_1^+$, 

        \item[(iv)] $w_{u, \infty}$ and $w_{v, \infty}$ solve the problem \eqref{eqn:ViscousOneSidedTwoMembranePb} in $B_1^+ \cap A$, for all $A \Subset B_1$,

        \item[(v)] any nontrivial blow-up between $w_{u, \infty}$ and $w_{v, \infty}$ is homogeneous of degree $\lambda$ in $B_1^+$, with
        \begin{equation}\label{eqn:FrequencyBoundBelow3/2}
        \lambda \coloneqq \lim_{r \to 0^+} \tilde{N}(r) \ge \frac{3}{2} .   
        \end{equation}
    \end{itemize}
\end{lemma}
\begin{proof}
    Point (i) follows from the uniform bound of $w_{u, n}$ and $w_{v, n}$ in $H^1\left(B_1^+\right)$, which follows from hypotheses \eqref{eqn:FrequencyCompactnessHp} and (e.g.) an inequality of type \eqref{eqn:FrequencyCompactnessProof1}.

    Point (ii) is a consequence of Lemma \ref{lemma:FrequencyCompactnessBlowUp} (ii), thanks to the compact embedding in Hölder spaces.

    Point (iii) follows e.g. from \eqref{eqn:FrequencyCompactnessHp} and the compact embedding of $H^1\left(B_1^+\right)$ in $L^2\left(\partial B_1^+\right)$.

   Point (iv) comes e.g. from section \ref{subsec:LimitPb}, or more simply, using point (ii).

   Now we turn to point (v). Concerning the homogeneity, we only deal with the case in which $w_{u, \infty}$ and $w_{v, \infty}$ are both nontrivial (otherwise the argument simplifies). Notice that, in such case, 
   \begin{equation}\label{eqn:FrequancyBlowUpNonDegBoundary}
       \int_{B_r^+} w_{u, \infty}^2 > 0 \quad \text{and} \quad \int_{B_r^+} w_{v, \infty}^2 > 0 \quad \text{for all } r \in (0, 1).
   \end{equation}
   Indeed, suppose by contradiction that for some $r \in (0, 1)$ condition \eqref{eqn:FrequancyBlowUpNonDegBoundary} does not hold. Hence, without loss of generality we can suppose that $w_{u, \infty} = 0$ on $\partial B_r \cap B_1^+$. By the uniqueness of solutions for problem \eqref{eqn:ViscousOneSidedTwoMembranePb}, necessarily $w_{u, \infty} = 0$ in $B_r^+$, and by unique continuation $w_{u, \infty} = 0$ in $B_1^+$. This gives a contradiction to the hypothesis that both blow-ups are nontrivial.
   
   For notational convenience let us denote for $r \in (0, 1)$
   \begin{align*}
       & E_{\infty}(r) \coloneqq \frac{1}{r^{N-2}} \int_{B_1^+ \cap B_r} \left( \Lambda_u \vert \nabla w_{u, \infty} \vert^2 + \Lambda_v\vert \nabla w_{v, \infty} \vert^2 \right) , \\
       & H_{\infty}(r) \coloneqq \frac{1}{r^{N-1}} \int_{ B_1^+ \cap \partial B_r} \left( \Lambda_u w_{u, \infty}^2 + \Lambda_v w_{v, \infty}^2 \right) ,
   \end{align*}
   so that
   \[
   \frac{d}{dr} H_{\infty}(r) = \frac{2}{r} E_{\infty}(r) .
   \]
   Using \eqref{eqn:HTildeHDiffBound}, \eqref{eqn:FrequencyCompactnessHp} and point (ii) we have
   \[
    N_{\infty}(r) \coloneqq \frac{E_{\infty}(r)}{H_{\infty}(r)} = \lim_{n \to +\infty} \frac{rr_n}{2} \frac{\frac{d}{dr}\tilde{H}(r r_n)}{\tilde{H}(rr_n) + O((rr_n)^{1+3\beta})} = \lim_{r \to 0^+} \tilde{N}(r) = \lambda ,
   \]
   where the limit exists thanks to Proposition \ref{prop:FrequencyAlmostMonotone}. Hence, $N_{\infty}(r)$ is constant for $r \in (0, 1)$. We can use this fact to prove the homogeneity of both $w_{u, \infty}$ and $w_{v, \infty}$.

   By a direct computation (similarly as in Lemma \ref{lemma:FrequencyBound2ndDeriv}) and using the boundary conditions of problem \eqref{eqn:ViscousOneSidedTwoMembranePb}, we have
   \begin{align*}
       & 0 = \frac{d}{dr} N_{\infty}(r) = \frac{H_{\infty(r)}\frac{d}{dr}E_{\infty}(r) - \frac{r}{2} \left( \frac{d}{dr} H_{\infty}(r) \right)^2}{H_{\infty}(r)^2} = \\
       & = \frac{2}{\left( r^{N-1} H_{\infty}(r) \right)^2} \Biggl[ \left( \int_{ B_1^+ \cap \partial B_r} \Lambda_u w_{u, \infty}^2 + \Lambda_v w_{v, \infty}^2 \right) \left( \int_{ B_1^+ \cap \partial B_r} \Lambda_u \left(\partial_r w_{u, \infty}\right)^2 + \Lambda_v \left(\partial_r w_{v, \infty}\right)^2 \right) + \\
       & - \left( \int_{ B_1^+ \cap \partial B_r} \Lambda_u w_{u, \infty} \partial_r w_{u, \infty} + \Lambda_v w_{v, \infty} \partial_r w_{v, \infty} \right)^2 \Biggr] .
   \end{align*}
   Let us denote
   \begin{align*}
       & A_{\infty}(r) \coloneqq \int_{ B_1^+ \cap \partial B_r} \Lambda_u w_{u, \infty}^2, \quad B_{\infty}(r) \coloneqq \int_{ B_1^+ \cap \partial B_r} \Lambda_v w_{v, \infty}^2 , \\
       & C_{\infty}(r) \coloneqq \int_{ B_1^+ \cap \partial B_r} \Lambda_u \left(\partial_r w_{u, \infty}\right)^2 , \quad D_{\infty}(r) \coloneqq \int_{ B_1^+ \cap \partial B_r} \Lambda_v \left(\partial_r w_{v, \infty}\right)^2 , \\
       & F_{\infty}(r) \coloneqq \int_{ B_1^+ \cap \partial B_r} \Lambda_u w_{u, \infty} \partial_r w_{u, \infty} , \quad G_{\infty}(r) \coloneqq \int_{ B_1^+ \cap \partial B_r} \Lambda_v w_{v, \infty} \partial_r w_{v, \infty} .
   \end{align*}
   Hence,
   \begin{align*}
       0 = & \left[ A_{\infty}(r) C_{\infty}(r) - F_{\infty}(r) \right] + \left[ B_{\infty}(r) D_{\infty}(r) - G_{\infty}(r) \right] + \\
       & + \left[ A_{\infty}(r) D_{\infty}(r) + B_{\infty}(r) C_{\infty}(r) - 2 F_{\infty}(r) G_{\infty}(r) \right] .
   \end{align*}
   Since each term in the square brackets is non-negative by the Hölder inequality, they all have to vanish. Hence, by the characterization of the identity cases in Hölder's inequality, by the first and second pair of brackets we have that there exist two functions $f(r), g(r) : (0, 1) \to \R$ such that
   \[
   \partial_r w_{u, \infty} = f(r) w_{u, \infty} \quad \text{and} \quad \partial_r w_{u, \infty} = g(r) w_{u, \infty} \quad \text{for all } r \in (0, 1),
   \]
   while from the last pair of brackets, thanks to \eqref{eqn:FrequancyBlowUpNonDegBoundary} we have 
   \[
   f(r) = g(r).
   \]
   Using the constancy of $N_{\infty}(r)$, we obtain 
   \[
   f(r) = g(r) = \frac{\lambda}{r}
   \]
   Since $w_{u, \infty}(0) = w_{v, \infty}(0) = 0$, both functions are homogeneous of degree $\lambda$.

   We are only left to prove \eqref{eqn:FrequencyBoundBelow3/2}. To this aim, let us first consider the case in which only one of the two blow-ups, e.g $w_{u, \infty}$, is nontrivial. In such case, by \eqref{eqn:ViscousOneSidedTwoMembranePb} the function $w_{u, \infty}$ solves a problem with homogeneous Neumann boundary conditions on $\partial B_1^+ \cap \{ x_N = 0 \}$. Moreover, by point (ii) we have $\nabla w_{u, \infty}(0) = 0$. Hence, by elliptic regularity necessarily $\lambda \ge 2$.

   Now suppose that $w_{u, \infty}$ and $w_{v, \infty}$ are both nontrivial, and let us consider the functions
   \[
   w_1 \coloneqq w_{u, \infty} - w_{v, \infty} \quad \text{and} \quad w_2 \coloneqq \Lambda_u w_{u, \infty} + \Lambda_v w_{v, \infty} .
   \]
   By \eqref{eqn:ViscousOneSidedTwoMembranePb} the function $w_1$ is a solution of the thin obstacle (or Signorini) problem, while $w_2$ is again a solution of a problem with homogeneous Neumann boundary conditions on $\partial B_1^+ \cap \{ x_N = 0 \}$. Also in this case, by point (ii) we have $\nabla w_1(0) = \nabla w_2(0) = 0$. Since $w_1$ and $w_2$ cannot both vanish identically (otherwise $w_{u, \infty} = w_{v, \infty} = 0$), $\lambda$ is the minimum between the lowest homogeneity for the two problems, which coincides with $3/2$, namely the lowest homogeneity of the Signorini problem (see e.g. \cite[Section 5]{Fernandez-Real:ThinObstacleSurvey}).
\end{proof}

\subsection{Proof of the sharp regularity}\label{subsec:SharpRegularity}
Using Lemma \ref{lemma:FrequencyCompactnessBlowUp}, Proposition \ref{prop:FrequencyAlmostMonotone} and Lemma \ref{lemma:FrequencyBdFromBelow}, we are now going to prove Theorem \ref{thm:SharpRegularity}.

\begin{proof}[Proof of Theorem \ref{thm:SharpRegularity}.]
    We will prove a pointwise (but uniform) $C^{1, 1/2}$ behavior of the solutions $u$ and $v$ at regular branching points. Once this is achieved, Theorem \ref{thm:SharpRegularity} follows by a projection argument onto such points (for a similar argument see e.g. \cite{ChangLaraSavin:BoundaryRegularityOnePhase} or \cite{FerreriVelichov2023:BernoulliInternalInclusion}).

    Let $x_0 \in \partial \{ u>0 \} \cap \partial \{ v>0 \}$ be a regular point of the free boundaries. Let us denote
    \[
    \tilde{w}_{u, r} \coloneqq \frac{w_u(rx)}{r^{3/2}} \quad \text{and} \quad \tilde{w}_{v, r} \coloneqq \frac{w_v(rx)}{r^{3/2}} , \quad r \in (0, 1) .
    \]
    Taking into account analogous considerations to those leading to \eqref{eqn:FrequancyBdBelowHp2}, we can suppose that both $\tilde{w}_{u, r}$ and $\tilde{w}_{v, r}$ are defined over $B_1^+$.
    
    Our aim is to show that (see Remark \ref{rmk:FrequencyUniformConstantsNeigborhood} for the dependence of the constants on the point $x_0$)
    \begin{equation}\label{eqn:SharpRegToProve}
        \exists \, C, c>0 : \| \tilde{w}_{u, r} \|_{L^{\infty}(B_1^+ )} + \| \tilde{w}_{v, r} \|_{L^{\infty}(B_1^+ )} \le C \quad \text{for all } r \in (0, c).
    \end{equation}
    To begin with, we observe that by Proposition \ref{prop:FrequencyAlmostMonotone} and Lemma \ref{lemma:FrequencyBdFromBelow} (and an integration in $r$) we have
    \[
    \tilde{H}(r) \le C r^{3} \text{ for } r \in (0, c).
    \]
    Since by Proposition \ref{prop:FrequencyAlmostMonotone} the frequency function is almost non-decreasing (hence bounded from above by its value for some fixed $r_0$), \eqref{eqn:SharpRegToProve} follows from Remark \ref{rmk:PrecompactH1/2Tor^3/2} and Lemma \ref{lemma:FrequencyCompactnessBlowUp} (i).
\end{proof}

\subsection*{\bf Acknowledgments}
L.F. and B.V. are supported by the European Research Council (ERC), EU Horizon 2020 programme, through the project ERC VAREG - \it Variational approach to the regularity of the free boundaries \rm (No. 853404). L.F. is also a member of INDAM-GNAMPA.


\bibliographystyle{plain}
\bibliography{feve2023.bib}


\appendix

\medskip
\small
\begin{flushright}
\noindent 
\verb"lorenzo.ferreri@sns.it"\\
Classe di Scienze, Scuola Normale Superiore\\ 
Piazza dei Cavalieri 7, 56126 Pisa (Italy)
\end{flushright}

\begin{flushright}
\noindent 
\verb"bozhidar.velichkov@unipi.it"\\
Dipartimento di Matematica, Università di Pisa\\ 
Largo Bruno Pontecorvo 5, 56127 Pisa (Italy)
\end{flushright}

\end{document}